\documentclass[12pt]{amsart}
\usepackage{amssymb}
\usepackage{amsmath}
\usepackage{color}

\definecolor{orange}{rgb}{1,0.5,0}

\def\Z{{\mathbb Z}}\def\T{{\mathbb T}}\def\R{{\mathbb R}}\def\C{{\mathbb C}}

\def\CC{{\mathcal C}}\def\LL{{\mathcal L}}

\def\EE{{\mathcal E}}

\def\OO{{\mathcal O}}
\def\MM{{\mathcal M}}\def\PP{{\mathcal P}}

\def\bC{{\mathbb C}}

\def\cC{{\mathcal C}}\def\cL{{\mathcal L}}

\def\cT{{\mathcal T}}\def\cE{{\mathcal E}}
\def\cN{{\mathcal N}}\def\cD{{\mathcal D}}

\def\cM{{\mathcal M}}\def\cP{{\mathcal P}}

\def\Jac{\rm Jac}
\def\mes{\rm mes}

\def\f{\varphi}\def\o{\omega}\def\e{\eta}
\def\g{\gamma}
\def\s{\sigma}\def\k{\kappa}
\def\al{\alpha}\def\be{\beta}\def\ep{\varepsilon}
\def\be{\beta}\def\t{\tau}\def\la{\lambda}\def\m{\mu}
\def\th{\theta}

\def\L{\Lambda}\def\G{\Gamma}

\def\i{\infty}
\def\p{\partial}

\def\Leb{\mathrm{Leb}}

\def\b#1{\lbrace#1\rbrace}

\def\aa#1{\left\Vert#1\right\Vert}

\def\l#1{\langle #1\rangle}

\def\<{\langle}
\def\>{\rangle}

\def\sbs{\subset}

\def\sm{\setminus}
\def\lsim{\lesssim}

\def\Lin{\mathrm{Lin}}

\def\w{\mathrm{anal}}

\theoremstyle{plain}
\newtheorem{Thm}{Theorem}[section]

\newtheorem{Prop}[Thm]{Proposition}
\newtheorem{Lem}[Thm]{Lemma}
\newtheorem{Cor}[Thm]{Corollary}

\newtheorem*{Rem}{Remark}

\newtheorem*{Cor*}{Corollary}
\newtheorem*{Prop*}{Proposition}

\def\cC{\mathcal C}

\newcommand{\ph}{\varphi}

\def\a{\alpha}
\def\w{\omega}

\def\ti{\tilde}
\def\e{\varepsilon}
\def\pa{\partial}
\def\s{\sigma}
\def\th{\theta}

\newcommand{\ntop}[2]{\genfrac{}{}{0pt}{1}{#1}{#2}}

\let\newpf\proof \let\proof\relax

\newcommand{\ba}{\overline{A}}

\newcommand{\cO}{\mathcal{O}}

\newcommand{\cR}{\mathcal{R}}

\def\be{\begin{equation}}
\def\ee{\end{equation}}

\def\ba{{\begin{align}}}
\def\ea{{\end{align}}}

\def\bm{\begin{matrix}}
\def\em{\end{matrix}}

\def\a{{\alpha}}

\def\cL{{\mathcal L}}

\def\NR{\mathcal{N}\mathcal{R}}
\def\und{\underline}

\def\g{{\gamma}}

\def\bD{\mathbb{D}}
\def\bN{\mathbb{N}}
\def\0{{\mathbf 0}}

\newtheorem{thm}{Theorem}[section]

\newtheorem{lemma}[thm]{Lemma}

\theoremstyle{remark}

\def\cO{\mathcal{O}}

\def\cQ{\mathcal{Q}}

\def\cC{\mathcal{C}}

\def \bn {\hfill \\ \smallskip\noindent}

\theoremstyle{definition}

\newenvironment{proof}{ \noindent{\it Proof.}\quad}{\ \hfill $\Box$\vskip .2cm}

\def\sm{\smallsetminus}

\newcommand{\id}{\operatorname{id}}

\newcommand{\eps}{{\epsilon}}
\newcommand{\De}{{\Delta}}
\newcommand{\de}{{\delta}}

\newcommand{\si}{{\sigma}}

\newcommand{\om}{{\omega}}

\newcommand{\N}{{\mathbb N}}

\newcommand{\bR}{{\mathbb R}}
\newcommand{\bT}{{\mathbb T}}

\def\B0{{\bold{0}}}

\def\be{\begin{equation}}
\def\ee{\end{equation}}

\catcode`\@=12

\def\Empty{}
\newcommand\oplabel[1]{
  \def\OpArg{#1} \ifx \OpArg\Empty {} \else
  	\label{#1}
  \fi}
		
\newcommand{\comm}[1]{}
\newcommand{\comment}[1]{}

\newenvironment{dedication}
  {
   \thispagestyle{empty}
   \vspace*{\stretch{1}}
   \itshape             
   \raggedleft          
  }
  {\par 
   \vspace{\stretch{3}} 
  }

\begin{document}

\title[KAM-tori near an analytic elliptic fixed point ]
{KAM-tori near an analytic elliptic fixed point }
\author{L. H. Eliasson, B. Fayad, R. Krikorian}
\address{
IMJ-PRG University Paris-Diderot\\
IMJ-PRG CNRS\\
LPMA UPMC}
\email{hakan.eliasson@math.jussieu.fr, bassam@math.jussieu.fr, raphael.krikorian@upmc.fr }

\date{\today}

\thanks{Supported by  ANR-10-BLAN 0102 and  ANR-11-BS01-0004}

\maketitle

\begin{dedication}
Dedicated to our friend Alain Chenciner
\end{dedication}

\begin{abstract} 

$ \ $ 
We study the accumulation of an elliptic fixed point of a real analytic Hamiltonian   by  quasi-periodic invariant tori.  


We show that a fixed point   with Diophantine frequency vector  $\o_0$ is always accumulated by  invariant  complex analytic KAM-tori. Indeed,  the following alternative holds:   If the  Birkhoff normal form of the Hamiltonian  at the invariant point  satisfies a R\"ussmann transversality condition, the  fixed point  is accumulated by real analytic KAM-tori which cover positive Lebesgue measure in the phase space (in this part it suffices to assume that $\o_0$ has rationally independent coordinates). If the Birkhoff normal form is degenerate,  there exists an analytic subvariety   of complex dimension at least $d+1$  passing through 0   that is foliated by  
 complex analytic KAM-tori with frequency $\omega_0$. 
 
This is an extension of previous results obtained in \cite{EFK} to the case of an elliptic fixed point.
\end{abstract}

\tableofcontents 

\section{Introduction}\label{sIntro}

 Let
$\o_0\in\R^d$  and let
$$
 (*) \quad \left\{\begin{array}{l}
H(x,y)=\l{\o_0,r} +\cO^3(x,y)\\
r=(r_1,\dots,r_d),\quad r_j=\frac12(x_j^2+y_j^2)
\end{array}\right. $$
be a 
real analytic function defined in a neighborhood of $(0,0)$. The Hamiltonian system associated to $H$ is given by the vector field $X_H=(\pa_{y}H, -\pa_{x}H)$, namely
$$  \quad\left\{\begin{array}{l}
\dot x=\pa_{y}H(x,y)\\ 
\dot y=-\pa_{x}H(x,y).\end{array}\right.$$
The flow of $X_H$ has a fixed point $\cP_0=\b{(0,0)}$. 
We are interested in the study of whether this Hamiltonian system admits, besides $\cP_{0}$, other  invariant sets. More precisely, we shall try to find real analytic {\it KAM-tori for $X_H$} in a neighborhood of $\cP_{0}$, that is,  real analytic Lagrangian  tori invariant under $X_H$ on each one of  which the flow of $X_H$ is conjugated to a translation flow  $\ph\mapsto \ph+t \omega$; it is usually required (and we shall follow this requirement)  that $\omega\in\T^d$, the {\it frequency vector}, is in some Diophantine set $DC(\kappa,\tau)$ ($\kappa,\tau>0$)  defined by the property
\begin{equation} |\<k,\o\>|\ge \frac{\k}{|k|^{\t}}\quad \forall k\in\Z^d\sm \{0\}. \label{cd} \end{equation}
We will say that $\o_0 \in \R^d$  is irrational when its coordinates are rationally independent.

We call the complexification of  a real analytic KAM-torus  a {\it complex analytic KAM-torus for $X_H$}, that is, a complex analytic Lagrangian toric manifold invariant  under (the complexification of) $X_H$ on which the flow is conjugated  to a translation flow  $\ph\mapsto \ph+t \omega$. Note that there are complex analytic KAM-tori that are not the complexification of any real KAM-torus. 
Invariant complex analytic toric manifolds  were studied in different settings related to KAM theory (see for example \cite{Sto}).

Notice that the problem of finding real or complex analytic KAM-tori in a neighborhood of the invariant fixed point $\cP_{0}$ can be given various more or less strong forms. For example, one can ask for  finding a  set of KAM  tori whose Lebesgue density in the phase space tend to one in the neighborhood of $\cP_0$. We shall call this KAM stability.

 In classical KAM theory, an elliptic fixed point $\cP_0$ is shown to be KAM-stable under the hypothesis that $\o_0$ is irrational (or just sufficiently non resonant) and that $H$ satisfies  a Kolmogorov non degeneracy condition of its Hessian matrix at $\cP_0$. Further development of the theory allowed to relax the non degeneracy condition. In this paper we prove KAM stability of $\cP_0$ under the R\"ussmann transversality condition on the Birkhoff normal form of $H$ at $\cP_0$.

We note that for non singular perturbative theory of analytic Hamiltonians it is known that the R\"ussmann condition is necessary and sufficient for KAM stability -- survival after perturbation of a positive measure set of KAM-tori -- of analytic integrable Hamiltonian systems (see  \cite{R} and \cite{Sev}).  We stress however that the study of the dynamics in a neighborhood of an elliptic fixed point, or near a given invariant torus, is a singular perturbation problem and that, therefore, the latter results  do not apply {\it per se}.

\medskip 
 
 The problem is more tricky if no nondegeneracy conditions are imposed on the Hamiltonian. In the analytic setting, no examples are known of an elliptic fixed point $\cP_0$ with $\omega_0$ irrational that is not KAM stable.  
It was conjectured by M. Herman in his ICM98-lecture \cite{H}
that for analytic Hamiltonians, KAM stability holds in the neighborhood of a  KAM  torus $\cT_0$ or of an elliptic fixed point $\cP_0$ if their frequency is assumed to be Diophantine. The conjecture is known to be true in two degrees of freedom $d=2$ \cite{R}, but remains open in general.

In this paper, we show that  a fixed point   with Diophantine frequency vector  $\o_0$ of an analytic Hamiltonian is always accumulated by complex analytic KAM-tori.   We previously obtained a similar  result in the neighborhood of an invariant analytic torus with Diophantine frequency vector \cite{EFK}. In the latter setting, the tori obtained were real analytic, but in the context of elliptic fixed points our method does not necessarily yield real analytic tori.

The current paper follows the same strategy  as in \cite{EFK}Êand provides the necessary modifications required by the absence of nice action-angle coordinates in the neighborhood of the  fixed point. 

An advantage however of the elliptic fixed point case, compared to that of an invariant torus, is that the Birkhoff normal form can be defined and conjugations up to any order can be performed under the sole condition that $\o_0$  is irrational. This is why we obtain the KAM stability of any irrational fixed point under the  R\"ussmann transversality condition, a result that we could not obtain for an invariant torus with irrational frequency, except in $2$ degrees of freedom (see  \cite{EFK}, Sections 2 and 9).

\subsection{Statement of the result}

Our main theorem is the following.
\begin{Thm}\label{theo:1.1}  Let $H:(\R^{2d},0)\to \R$ be a real  analytic function of the form $(*)$
and assume that  $\omega_0$ is Diophantine. Then, the origin is accumulated by infinitely many complex  analytic 
KAM-tori for $X_H$.\end{Thm}

 Let $N_H$ be the {\it Birkhoff Normal Form} of $H$ --  for the Birkhoff Normal Form
at a Diophantine, and more generally an irrational elliptic equilibrium, one can consult for example
\cite{SM}.

We say that $N_H$ is $j$-{\it degenerate} if there exist
$j$ orthonormal vectors $\g_1,\dots,\g_j$ such that for every   $r\sim 0\in\bR^d$
$$\l{\p_rN_H(r),\g_i}=0\quad \forall\ 1\le i\le j,$$
but no $j+1$ orthonormal vectors with this property. Since
$\omega_0\not=0$ clearly $j\le d-1$. 
{ A $0$-degenerate $N_H$ is 
also said to be {\it non-degenerate}.}

Our Main Theorem  is the consequence of  Theorems \ref{mA} and \ref{mB} below.

\begin{Thm}\label{mA}   Let $H:(\R^{2d},0)\to \R$ be a real  analytic function of the form $(*)$
and assume that  $\omega_0$ is Diophantine. If $N_H$ is $j$-degenerate,
then there exists an analytic subvariety containing $0$ of complex dimension $d+j$ 
 foliated by invariant complex analytic KAM-tori for $X_H$  with  translation vector $\o_0$.
\end{Thm}

A stronger result is known when $N_H$ is $(d-1)$-degenerate. Indeed 
R\"ussmann \cite{R} (in a different setting) proved

\begin{Thm}\label{Russmann}
{If  $\o_0$ is Diophantine and} $N_H$ is $(d-1)$-degenerate,
then a full neighborhood of
$0 \in\bR^{2d}$ is foliated by real analytic KAM-tori  for $X_H$ with translation vector in $\R\o_0$.
\end{Thm}

Our proof of Theorem \ref{mA} (see Section \ref{sec:mA}) will also yield R\"ussmann's result. 

In the nondegenrate case we will prove the following.

\begin{Thm}\label{mB}  Let $H:(\R^{2d},0)\to \R$ be a real  analytic function of the form $(*)$
and assume that  $\omega_0$ is irrational. 
If  $N_{H}$ is non-degenerate, then  in any neighborhood of
$0\in\bR^{2d}$ the set of real analytic KAM-tori for $X_H$ is of
positive Lebesgue measure and density one at $0$. 
\end{Thm}

\medskip

\subsection{Strategy of the proof}\label{strategy}

 We adopt a similar strategy to the one of \cite{EFK} that was inspired by previous works of Herman and Moser. The basis is a counter term KAM-theorem in which a Hamiltonian as in $(*)$ is conjugated, for any action value $c \sim 0\in\R^d$  and any $\omega$ in some fixed Diophantine class, to a Hamiltonian that has an invariant torus at the  action equal to $c$ with frequency $\omega$ up to a correction term $\Lambda(c,\o)$. Furthermore, for every value $c \in \R^d$ in the neighborhood of $0$ of the action variable there exists a unique frequency $\Omega(c)$ that   cancels the counter term : $\Lambda(c,\Omega(c))=0$. We call the map $c \mapsto \Omega(c)$ the {\it frequency map}.  If $\Omega(c)$ is Diophantine this yields an invariant KAM-torus with frequency $\Omega(c)$. One can show that the jets of the function $\Omega(c)$ are given by those of the gradient of the Birkhoff normal form when the latter is well defined (which is the case if $\o_0$ is irrational since we are dealing with fixed points).
The following alternative then holds : either the BNF is non degenerate and the function $\Omega$ takes Diophantine values on a set of Lebesgue density $1$ at $\cP_0$, which yields KAM stability; or the BNF is degenerate and  we can use the analytic dependance of the counter term on the action variable to show the existence of a direction (after a coordinate change in the action variable) that spans a complex analytic subvariety foliated by  complex analytic KAM-tori with  translation vector $\o_0$.

Let us briefly explain why we do not necessarily obtain real analytic tori  by our method in this context of elliptic fixed point.  In the normal form expression we look for a 
change of variable $Z$ defined 
 in a neighborhood of  the origin containing  the torus $\cT:=\{x_{j}^2+y_{j}^2=c_{j}^2, j=1,\ldots d\}$  such that 
\be H\circ Z(x,y)=\Gamma+\sum_{i=1}^d\frac{\omega_{j}}{2}(x_{j}^2+y_{j}^2-c_{j}^2) +\sum_{j=1}^d(x_{j}^2+y_{j}^2-c_{j}^2)F_{j}(x,y)\label{hameq}\ee
where $F=(F_{1},\ldots,F_{d})$ is null on $x^2+y^2-c^2=0$ ($\Gamma$ is a constant which is unimportant). The torus $\cT$ is then invariant by the flow of $H\circ Z$ (hence $Z^{-1}\cT$ is invariant by $H$).

Since $H$ is real analytic it has   a holomorphic extension to a polydisk $\bD_{\rho}^{2d}\subset \bC^{2d}$. Notice that  the extension of $H$ to $\bD_{\rho}^{2d}$ thus satisfies  $\overline{H(x,y)}=H(\bar x,\bar y)$.
It will be convenient  to make the following change of variables: for $x,y\in\bC^d$, define $z_{j}=\frac{1}{2}(x_{j}+\sqrt{-1}y_{j})$, $w_{j}=\frac{1}{2}(x_{j}-\sqrt{-1}y_{j})$ so that $r_{j}:=\frac{1}{2}(x_{j}^2+y_{j}^2)=z_{j}w_{j}$, $j=1,\ldots,d$.
One has $dz\wedge dw=-\sqrt{-1}dx\wedge dy$. 
Notice that a function $(x,y)\mapsto f(x,y)$ is real analytic (hence satisfies  $\overline{f(x,y)}=f(\bar x,\bar y)$) if and only if $\ti f(z,w):= f(x,y)$ satisfies the symmetry $\overline{\ti f(z,w)}=\ti f(\bar w, \bar z)$. We then  say that $\ti f$ is 
$\sigma$-{\it symmetric}  (where $\sigma$ is the involution $\sigma(z,w)=(\bar w,\bar z)$).
 If $f$ depends  real holomorphically  on an extra complex parameter $c$ and smoothly on an extra real parameter $\omega$,  then $\ti f :\bD^d_{\rho}\times\bD^d_{\rho}\times \bD^d_{\delta}\times B(0,1)\to\C$ satisfies
$\overline{{\ti  f}(z,w,c,\omega)}=\ti f(\bar w, \bar z,\bar c,\omega)$, i.e. it is symmetric with respect to the 
involution  $\sigma(z,w,c)=(\bar w,\bar z,\bar c)$. By a slight abuse of notation we shall call also
this property $\sigma$-{\it symmetry}. One can define a similar notion of $\s$-symmetry for diffeomorphisms (see Section \ref{sec:1.4}).

Equation (\ref{hameq}) is then equivalent to finding $\omega\in\R^{2d}$, $c\in\R^{d}$, an exact  symplectic change of coordinates $\ti Z$ for $dz\wedge dw$ and maps $\ti F_{j}$  null on $zw-c^2=0$ such that $\ti Z$ and the $\ti F_{j}$ are $\sigma$-symmetric and 
$$ \ti H\circ \ti Z(z,w)=\Gamma+\sum_{i=1}^d\omega_{j}(z_{j}w_{j}-c_{j}^2) +\sum_{j=1}^d(z_{j}w_{j}-c^2)\ti F_{j}(z,w).$$
The searched for torus $\cT$ then corresponds in the $(z,w)$-coordinates to  $\{z_{j}w_{j}=c_{j}^2,j=1,\ldots,d\}\cap \{(z,w)\in \bC^{2d}:\sigma(z,w)=(z,w)\}$.

The strategy of the proof is then to find for some values of $c\in\R^{d}$ and $\omega\in\R^{d}$ such a normal form. 

However, in the $j$-degenerate case ($j\ne 0$), it will only be possible to do so for some $c^2:=(c_{1}^2,\ldots,c_{d}^2)$ (but not necessarily $c$ itself)   in $\R^{d}$ and consequently it will not be possible to ensure that the searched for  tori $\cT$ are {\it real}.  We obtain instead  complex analytic KAM-tori for $X_H$.

\subsection{Notations}\label{sec:1.4}

We denote by $\bD_{\delta}^d$ the polydisk in $\bC^d$
with radius $\de$. More generally if $d=(d_1,\dots,d_n)$ and 
$\de=(\de_1,\dots,\de_n)$, then
$$\bD^d_\de=\bD^{d_1}_{\de_1}\times\dots\times  \bD^{d_n}_{\de_n}.$$

Let  $f: \bD^e_\de\to\bC$ be a holomorphic function.
We denote by $\p_{z_i} f$ the partial derivate
of $f$ with respect to $z_i$ and we use the usual multi-index
notation like $\p_{z}^{\al} f$.  If $z=(z',z'')$ we say that
$$f\in \OO^j(z')$$
if and only if $\p_{z'}^{\al'} f(\f,0,z'')=0$ for all $|\al'|<j$.
We shall also use the same notations for $\bC^n$-valued functions
$f=(f_1,\dots, f_n)$ with the absolute value replaced
by $|f|=\max_i |f_i|$ (or some other norm on $\C^n$).

\medskip

 {\it $\s$-symmetry.}
Let $\s$ be the involution $(z,w,c)\mapsto (\bar w,\bar z,\bar c)$ on $\bC^d\times\bC^d\times\bC^d$.
A holomorphic function $f: \bD_\de^d\times\bD_\de^d\times\bD_\de^d\to\bC$ is 
$\s$-{\it symmetric} if, and only if,  $f\circ \s(z,w,c)=\overline{ f(z,w,c)}$. 
This means precisely that it takes real values on
the subspace $\b{(z,w,c)=\s(z,w,c)}$. A (local) 
mapping $F$ preserves this subspace if  and only if 
$$\s^{-1}\circ F\circ \s=F$$
--  we say then  that $F$ is $\s$-symmetric.

We let $\cC^{\o,\sigma}( \bD_\de^d\times\bD_\de^d\times\bD_\de^d)$ denote the space
of $\s$-symmetric holomorphic functions $ f:\bD_\de^d\times\bD_\de^d\times\bD_\de^d\to\bC$
provided with the norm
$$|f|_{\de}=\sup_{z\in \bD^e_\de}|f(z)|.$$

\medskip

{\it Formal power series.} Let $z=(z_1,\dots,z_n)$. An element 
$$f\in \C[[z]]$$
is a formal power series
$$f=f(z)=\sum_{\al\in \bN^n} a_\al z^\al$$
whose coefficients $a_\al\in \C$ 
(possibly vector valued). The notion of $\s$-symmetry carries over
to this more general framework. We denote by
$$[f]_j(z)= \sum_{|\al|=j} a_\a z^\al,$$
the homogenous component of degre $j$, and
$$[f]^j=\sum_{i\le j}[f]_i.$$

\medskip

{\it Parameters.} Let $B$ be an open subset of some euclidean
space. Define 
$$\cC^{\w,\i}(\bD_\de^d\times\bD_\de^d\times\bD_\de^d,B)$$
(or for short $\cC^{\omega,\infty}_{\delta}$)
 to be the set of $\cC^{\i}$ functions (possibly vector valued)
$$f:  \bD_\de^d\times\bD_\de^d\times\bD_\de^d\times B\ni(z,\o)\mapsto f(z,\o)$$
such that for all $\o\in B$
\footnote{\ we apologize for the double use of $\o$}
$$f_\o: \bD_\de^d\times\bD_\de^d\times\bD_\de^d\ni (\o)\mapsto f(z,\o)$$
is a  holomorphic function.  If in addition, this map is $\s$-symmetric, we shall write $f\in \cC^{\omega,\sigma,\infty}(\bD_\de^d\times\bD_\de^d\times\bD_\de^d,B)$.  We define
$$||f||_{\de,s}=\sup_{|\al|\le s}|\p^\al_\o f_\o|_{\de}.$$

\bn 
{\it $(\kappa,\tau)$-{\it flat} functions.}
A $\cC^{\infty }$ function $f: \bD^d\times B\to \bC$, $(z,\omega)\mapsto f(z,\omega)$ is  $(\kappa,\tau)$-{\it flat}  
if,  for any set of indices $\a,\beta$,
$$\pa_{z}^\a\pa_{\om}^\beta f(z,\omega)=0$$
whenever  $\omega\in DC(\kappa,\tau)$.

\bn 
{\it Tensorial notations.} When  $(v_{1},\ldots,v_{m)}\mapsto  B(v_{1},\cdots,v_{m})$ is a $m$-multilinear form on a vector space $V$, we shall often see it as a linear form  on the $m$-th tensorial product $V^{\otimes m}$, and use the corresponding tensorial notations. Also, we denote by $\otimes_{sym}$ the symmetrized tensor product $v_{1}\otimes\cdots\otimes v_{m}=\sum_{\si}v_{\s(1)}\otimes\cdots\otimes v_{\s(m)}$ where the sum is on all the permutations of $\{1,\ldots,d\}$.

\bigskip

\section{Power series expansion}\label{sec:2}
\bigskip

\subsection{Expansion with   Non Resonant functions }
Let 
$$f(z,w)=\sum_{\a,\beta\in\N}f_{\a,\beta}z^\a w^\beta$$
be some holomorphic function defined on a polydisk of $(\C^2,0)$  --  or more generally a formal
power series. We have
$$f(z,w)=\sum_{n=0}^\infty (zw)^n\sum_{\a\beta=0}f_{\a,\beta}z^{\a}w^{\beta}$$ 
and since in the last sum in  the  previous expression $\a=0$ or $\beta=0$ we can find analytic $g_{n}$ and $h_{n}$, $n\in\Z$ such that
$$f(z,w)=\sum_{n=0}^\infty(zw)^n(g_n(z)+h_n(w)).$$
A similar procedure or a simple induction argument  show that if $f$ is now analytic in some polydisk  $\bD^{2d}_{\rho}$ of $(\C^{2d},0)$ then
$$f(z,w)=\sum_{\underline{n}\in\N^d}(z_{1}w_{1})^{n_{1}}\cdots(z_{d}w_{d})^{n_{d}}\sum_{(\a,\beta)\in\NR}f_{\underline{n},\a,\beta} z^\alpha w^\beta$$
where $\NR$ is the set of $(\a,\beta)\in(\N^d)^2$ such that  for all $i=1,\ldots,d$, $\a_{i}\beta_{i}=0$.
A power series of the form $h(z,w)=\sum_{(\a,\beta)\in\NR}h_{\a,\beta}z^\a w^\beta$ will be called {\it non-resonant} and we denote by $\widehat{\NR}$ the vector space  of all  non-resonant  functions.  Notice that we allow for the existence of constant terms in this definition. We can also say that 
$$f(z,w)=\sum_{\underline{n}\in\N^d}(z_{1}w_{1})^{n_{1}}\cdots(z_{d}w_{d})^{n_{d}}\sum_{\underline{\e}\in\{0,1\}^d}f_{\underline{n},\underline{\e}}(r_{1}^{\e_{1}},\ldots,r_{d}^{\e_{d}})$$
where $f_{\underline{n},\underline{\e}}$ are  holomorphic in $\bD_{\rho}^{2d}$
 and where we have used the notation $r_{i}^\e=z_{i}$ if $\e=0$ and $r_{i}^\e=w_{i}$ if $\e=1$. 

The following fact will be useful:
\begin{Lem}\label{NR}For any $(\a,\beta)\in (\N^d)^2$ there is  a {\it unique} decomposition of the form $(\a,\beta)=(\und{n},\und{n})+(\a',\beta')$ where $\und{n}\in\N^d$ and $(\a',\beta')\in\NR$ (this means that for any $i=1,\ldots,d$, $\a_{i}=\und{n}_{i}+\a_{i}'$ and $\beta_{i}=\und{n}_{i}+\beta'_{i}$).
\end{Lem}
\begin{proof} 

To prove the  existence of such a decomposition just  take $\und{n}_{i}=\min(\a_{i},\beta_{i})$, $i=1,\ldots,d$. To prove  uniqueness  we observe that  if for some $i$ $\und{n}_{i}\neq \und{\ti n}_i$, for example $\und{n}_{i}> \und{\ti n}_i$, then $\a'_i>\ti\a_i$ and $\beta'_i> \ti\beta_i$ a contradiction with $(\a',\beta')\in \NR$.

\end{proof}
The preceding discussion provides the following  decomposition 
\begin{Lem}\label{uniqueness}If $f(z,w)$ is holomorphic on some polydisk  $\bD_\de ^{d}\times \bD_\de ^{d}$ there exists a unique decomposition 
\be f(z,w)=\sum_{\underline{n}\in\N^d}(z_{1}w_{1})^{n_{1}}\cdots(z_{d}w_{d})^{n_{d}}\ti f_{\underline{n}}(z,w)\label{e*}\ee
where $\ti f_{\underline{n}}\in \widehat{\NR}$ are  holomorphic on $\bD_\de ^{d}\times \bD_\de ^{d}$  -- the series converges uniformly on compact sub domains of $\bD_\de ^{d}\times \bD_\de ^{d}$. Furthermore $f$ is $\sigma$-symmetric if and only if all the $\ti f_{\und{n}}$ are.
\end{Lem}
\begin{proof}
To prove uniqueness, one just have to prove that if $f$ is null, the same is true of all the series $f_{\und{n}}$. This is done by looking at the coefficients of the right hand side of (\ref{e*}) and by using Lemma \ref{NR}. The $\sigma$-symmetry of the $f_{\underline{n}}$ comes  from the uniqueness. 
\end{proof}
If now $f$ depends (or not) on a parameter $c=(c_{1},\ldots,c_{d})$, by writing $z_{j}w_{j}=c_{j}+(z_{j}w_{j}-c_{j})$ in (\ref{e*}) we get   an expansion uniformly  converging on small compact neighborhoods of 0:
\begin{equation}f(z,w,c)=\sum_{\underline{n}\in\N^d}(z_{1}w_{1}-c_{1})^{n_{1}}\cdots(z_{d}w_{d}-c_{d})^{n_{d}}f_{\underline{n}}(z,w,c).\label{A.exp}\end{equation}
We again notice that  each $f_{\underline{n}}(\cdot,\cdot,c)$ is non-resonant (for any fixed $c$) and 
\be f_{\underline{n}}(z,w,c)=\sum_{\underline{k}\geq \underline{n}}{\underline{k}\choose\underline{n}}c^{(\underline{k}-\underline{n})}\ti f_{\underline{k}}(z,w,c).\label{9candec}\ee
We shall still denote by $\widehat {\cN\cR}$ the set of functions $f(z,w,c)$ which are non-resonant for each fixed $c$.

\begin{Lem}\label{l10.4}If $f$ is $\sigma$-symmetric, there exists a unique  decomposition of the form (\ref{A.exp}) where each $f_{\underline{n}}$  is non-resonant and   $\sigma$-symmetric.
\end{Lem} 
\begin{proof}We have to prove that if in (\ref{A.exp}) $f$ is equal to 0 then all the $f_{\underline{n}}$ are null.  If $f_{\und{n}}(z,w,c)=\sum_{\a',\beta',\gamma'}f_{\und{n},\a',\beta',\gamma'}z^{\a'}w^{\beta'}c^{\gamma'}$, the coefficient of $z^\a w^\beta$ in  (\ref{A.exp})  is the sum 
$$\sum_{\ntop{\underline{k},\a',\beta',\gamma'}{(\underline{k},\underline{k})+(\a',\beta')=(\a,\beta)}}\sum_{\underline{n}\geq \underline{k}} {\underline{n}\choose \underline{k}}(-1)^{|\und{n}-\und{k}|}c^{\gamma'+(\underline{n}-\underline{k})}f_{\underline{n},\a',\beta',\gamma'}$$
Since in the last sum  $(\a',\beta')\in\NR$, the decomposition $(\a,\beta)=(\underline{k},\underline{k})+(\a',\beta')$ is unique by Lemma \ref{NR}, and thus the last sum is just (the summation is in $\underline{n}$)
$$\sum_{\ntop{\underline{n}\geq \underline{k}}{\gamma'}} {\underline{n}\choose \underline{k}}(-1)^{|\und{n}-\und{k}|}c^{\gamma'+(\underline{n}-\underline{k})}f_{\underline{n},\a',\beta',\gamma'}.$$
By assumption, for any $\und{k}$, any $(\a',\beta')\in\NR$,  this has to be equal to zero for any $c$ in a neighborhood of zero. Multiplying the last sum by $(d+c)^{\und{k}}$ and making the summation on all $\und{k}\geq 0$ one gets
$$\sum_{\ntop{\und{n}}{\gamma'}} c^{\gamma'}d^{\und{n}}f_{\underline{n},\a',\beta',\gamma'}=0.$$
This being true for all $c,d$ in a neighborhood of 0 one has $f_{\und{n}}=0$.

The $\sigma$-symmetry of the $f_{\underline{n}}$ comes  from the uniqueness. 

\end{proof}

\noindent{\bf Remark} Lemmas  \ref{uniqueness} and \ref{l10.4} hold in the case of formal series in $\C[[z,w,c]]$.
\bigskip

If $p\in\N$, we shall denote 
\begin{equation} f_{p}(z,w,c)=\sum_{|\underline{n}|=p}(zw-c)^{\underline{n}}f_{\underline{n}}(z,w,c)\label{A.4}\end{equation}

We shall use the following notations. We have seen that  $f(z,w,c)$ can be   written under the form
\begin{multline}f(z,w,c)=f^{(0)}(z,w,c)+\<f^{(1)}(z,w,c),(zw-c)\>+\\
\<(zw-c),f^{[2]}(z,w,c)(zw-c)\>\label{candec1}\end{multline}
or 
\begin{multline}f(z,w,c)=f^{(0)}(z,w,c)+\<f^{(1)}(z,w,c),(zw-c)\>+\\
\<(zw-c),f^{(2)}(z,w,c)(zw-c)\>+f^{[3]}(z,w,c)(zw-c)^{\otimes 3}\label{candec2}\end{multline}
where  $f^{(0)}(z,w,c)$ and $f_{j}^{(1)}(z,w,c)$ ($j=1,\ldots,d$), $f_{i,j}^{(2)}(z,w,c)$ ($1\leq i,j\leq d$)  are in $\widehat{\NR}$  and 
where the notations   $f^{(2)}_{ij}(z,w,c)$ and $f^{[2]}_{ij}(z,w,c)$ ($1\leq i,j\leq d$)  denotes respectively
the sums      $\sum_{\underline{n}=\lambda_{i}+\lambda_{j}}f_{\underline{n}}(z,w,c)$  and $\sum_{\underline{n}\geq \lambda_{i}+\lambda_{j}}(zw-c)^{\underline{n}-\lambda_{i}-\lambda_{j}}f_{\underline{n}}(z,w,c)$
with $\lambda_{k}\in\N^d$ ($k=i,j$) denoting
 the multiindex $\lambda_{k}(l)=\delta_{kl}$ (Kronecker's symbol); $f^{[(3)]}$ is defined similarly. 
We shall call the decompositions (\ref{candec1}) and (\ref{candec2}) the {\it canonical decomposition} of $f$ (up to order 2 or 3).

\begin{Lem}\label{uniquedecomp}If  \begin{multline*}f(z,w,c)=a_{0}(z,w,c)+\<a_{1}(z,w,c),(zw-c)\>+\cO((zw-c)^2)\end{multline*}
with $a_{0}$ and $a_{1}$ in $\widehat{\cN\cR}$ then $a_{0}=f^{(0)}$ and $a_{1}=f^{(1)}$. 
\end{Lem}
\begin{proof} Let us denote $g(z,w,c)$ the $\cO((zw-c)^2)$ of the statement of the lemma.  The function $g$ can be written 
$$g(z,w,c)=\sum_{\a\in\N^d,|\a|=2}h_{\a}(z,w,c)(zw-c)^{\a}$$ and each $h_{\a}$ can be decomposed $h_{\a}(z,w,c)=\sum_{\underline{n}\in\N^d} h_{\a,\underline{n}}(z,w,c)(zw-c)^{\underline{n}}$ where all the $h_{\a,\underline{n}}$ are in $\widehat{\cN\cR}$ and so $g(z,w,c)=\sum_{\underline{m}\in \N^d,|\underline{m}|\geq 2}(zw-c)^{\underline{m}}g_{\underline{m}}(z,w,c)$ where each $g_{\underline{m}}(z,w,c):=\sum_{\a+\underline{n}=m,|\a|=2,\underline{n}\in\N^d}h_{\a,n}(z,w,c)$ is non resonant. The uniqueness given by Lemma \ref{l10.4} concludes the proof. 
\end{proof}

\subsection{The operators $\cM$, $\cD$ and $\cD^\omega$}\label{subsec:2.2}

We now define the  operator $\cM$ by 
$$(\cM f)(z,w)=\sum_{\a=\beta}f_{\a,\beta}z^\alpha w^\beta$$
(diagonal terms). If $\cM f=f$ we say that $f$ is {\it diagonal}.

Observe that if $f(z,w,c)=\sum_{\underline{n}\in\N^d}(zw-c)^{\und{n}}f_{\underline{n}}(z,w,c)$ where all the $f_{\und{n}}$ are in $\widehat\NR$ then 
$$\cM f=\sum_{\underline{n}\in\N^d}(zw-c)^{\und{n}}f_{\underline{n}}(0,0,c).$$

Let us introduce  the following  differential operators
$$\cD_{i} f=(\pa_{z_{i}}f)z_{i}-(\pa_{w_{i}}f)w_{i},\qquad \cD f=(\cD_{1}f,\ldots,\cD_{d}f),$$
and if $\omega\in\R^d$
$$\cD^{\omega}=\<\omega,\cD\>=\omega_{1}\cD_{1}+\cdots+\omega_{d}\cD_{d}.$$
Notice that 
$$\cD_{i}(z^\a w^\beta)=(\a_{i}-\beta_{i})(z^\a w^\beta),\qquad \cD^{\omega}(z^\a w^\beta)=\<\omega,\a-\beta\>(z^\a w^\beta).$$
All these definitions extend to the case when $f(z,w,c)$ depends on $c$; the derivatives are taken w.r.t. $(z,w)$ and $c$ is then seen as a parameter.

Since 
$\cD(zw-c)^{\underline n}=0$, we observe that 
\begin{equation}\cD f_{p}=\sum_{|\underline{n}|=p}(zw-c)^{\underline{n}}\cD f_{\underline{n}}(z,w,c), \qquad \cD^{\omega}f_{p}=\sum_{|\underline{n}|=p}(zw-c)^{\underline{n}}\cD^{\omega} f_{\underline{n}}(z,w,c).\label{A.5}\end{equation}

Let us mention the following, easy to prove, but important properties:
\begin{Lem}\label{cohom}Let $f(z,w,c)$ be a formal series expansion.
\begin{enumerate}
\item \label{item1}If $f$ is diagonal, so are $z\pa_{z}f$, $w\pa_{w}f$, $\cD f$ and $\cD^\omega f$.
\item\label{item2} If  $[\cdot]_{j}$ denotes the homogeneous polynomial part of total degree $j$ in $(z,w,c)$, the operators $z\pa_{z}$, $w\pa_{w}$, $\cD$, $\cD^\omega$ commute with $[\cdot]_{j}$.
\item \label{item3}If $g$ is either  $z\pa_{z}f$, $w\pa_{w}f$, $\cD f$ or $\cD^\omega f$ then $[g]_{0}=0$.
\item\label{item'4} If $f$ is $\s$-symmetric then $\cM f$, $\sqrt{-1}\cD f$ and $\sqrt{-1}\cD^\omega f$ are $\s$-symmetric.
\item \label{item4} Assume that $\omega\in\R^d$ is irrationaland let $g\in\bC[[z,w,c]]$ be a
$\sigma$-symmetric power series. Then the equation 
$$\cD^\omega f=\sqrt{-1}g$$
has a $\sigma$-symmetric solution $f\in \bC[[z,w,c]]$ if and only if $\cM g=0$, and the solution $f$ is unique modulo the addition of any diagonal series expansion in $\bC[[z,w,c]]$. 
\end{enumerate}
\end{Lem}

\bigskip

\section{Formal Normal Forms}\label{sBNF}

\bigskip

\subsection{Exact symplectic mappings and generating functions}\label{s23}

  Let $Z:(z,w)\mapsto (z',w')$ be a holomorphic mapping of $(\C^{2d},0)$ endowed with the 
canonical symplectic form $dz\wedge dw$.  Since $d(zdw)=dz\wedge dw$, $Z$ is  {\it symplectic} if and only if the one-form $Z^*(zdw)-zdw$ is {closed}. By definition, $Z$ is said to be  {\it exact}, or {\it exact symplectic} if and only if the  one-form $Z^*(zdw)-zdw$ is exact. 
(Since we are on a simply connected domain symplectic implies exact symplectic.) Under general
conditions --  $Z$ is $\cC^1$-close to the identity mapping for example -- there exists a holomorphic 
function $f:(\C^{2d},0)\to \C$ such that
\be(z',w')=Z(z,w)
 \iff  \label{2.2} \begin{cases}
z'=z+\pa_{w'}f(z,w')\\
w=w'+\pa_{z}f(z,w')
\end{cases}  \ee
The construction of $f$ is the following: since  $Z$ is exact, there exists a holomorphic 
function $g:(\C^{2d},0)\to \C$ such that $Z^*(zdw)-zdw=dg$ and we define $f$ by
\be\label{2.3}
f(z,w')=g(z,w)-z(w'-w),\quad w'=\phi(z,w'),\ee
where $w'=\phi(z,w')$ is determined (by the implicit function theorem) from $(z',w')=Z(z,w)$.

A function $f$  like in \eqref{2.2} is called a {\it generating function} for  $Z$  --  it is unique up to an additive constant.
Conversely, any holomorphic  function $f:(\C^{2d},0)\to \C$ is under general conditions   --   $f$ is $\cC^2$-close to zero for example  --   the generating function for a unique exact holomorphic mapping. (This is a straight forward verification which can be found in most books on symplectic dynamics\slash geometry,  for example in \cite{SM}.)

If $Z$ depends holomorphically on some parameters $c$, then its generating function depends holomorphically on $c$, and conversely. This correspondence also preserves reality  --   $Z$ is real holomorphic if, and only if,  it's generating function is real holomorphic  --  but $\sigma$-symmetry is not preserved.

These properties carry over to the setting of (formal) mappings of the form
$$(*)\quad Z(z,w,c^2)-(z,w)\in \C[[z,w,c^2]]\cap\OO^2(z,w,c).$$

 \begin{Lem}\label{formcomp}
\ 
 \begin{itemize}
 
\item[(a)] The set of mappings of the form $(*)$ is a group under composition, and the set of (formally) exact mappings
 of the form $(*)$ is a subgroup.
  
\item[(b)] A (formally) exact mapping of the form $(*)$ has a unique (formal) generating
function of the form
$$f(z,w',c^2)\in \C[[z,w',c^2]]\cap \cO^3(z,w',c).$$

\item[(c)] Any function
$$f(z,w',c^2)\in \C[[z,w',c^2]]\cap \cO^3(z,w',c)$$
is the generating function of a unique (formally) exact mapping of the form $(*)$.

 \end{itemize}
 \end{Lem} 

\begin{proof} 
{ (a) is a direct computation on formal power series. 
(b) follows since a closed (formal) one-form is exact, which gives us a formal $g$: notice that ``low order'' terms 
(in $z,w,c$) of  $g$ do not depend on ``high order'' terms of $Z$. 
By truncating $g$ at some order $N$ we can apply  the formula \eqref{2.3}
which gives an $f_N$: notice that ``low order'' terms of $f_N$ do not depend on ``high order'' terms of $g$. 
Therefore this defines a formal generating function by letting $N\to\infty$.
(c) follows  by truncating $f$ (in $z,w,c$) at some order $N$ and and then define $Z_N$ by the formula \eqref{2.2}: notice that ``low order'' terms of $Z_N$ do not depend on ``high order'' terms of $f$. 
Therefore this defines a (formally) exact  mapping $Z$ by letting $N\to\infty$.
}
\end{proof}

\subsection{Formal Normal Forms}\label{s24}
Assume that $H(z,w)$ is a formal Hamiltonian in $\bC[[z,w]]$ of the form
$$H(z,w)=\<\om_0,zw\>+\OO^3(z,w)$$
with a vector $\om_0\in\R^d$ which is rationally independent.
 It is a classical result that there exist a unique $N\in\bR[[r]]$, the Birkhoff Normal Form of $H$, and  a (formally) exact 
 mapping of the form
$$(*)\quad Z(z,w)-(z,w)\in \C[[z,w]]\cap\OO^2(z,w)$$ 
 such that 
$$H\circ Z(z,w)=N(zw)= \<\om_0,zw\>+\OO^3(zw).$$
If $c\in\bC^{d}$ is an extra formal parameter one can write
$$H\circ Z(z,w)=N(zw)=N(c^2)+\<\nabla N(c^2),zw-c^2\>+\cO^2(zw-c^2).$$
The aim of the following proposition is to prove { that such a representation is unique.}

\begin{Prop}\label{bnfIIelliptic-zw}
If there exist  a  formal  series 
$$f(z,w',c^2)\in \bC[[z,w',c^2]]\cap {  \OO^3(z,w', c^2)}$$
and formal series $\Gamma(c^2),\Omega(c^2)\in\bC[[c]]$,
$$\G(c^2)=\<\omega_{0},c^2\>+\OO^2(c^2),\quad \Omega(c^2)=\o_0+\OO(c^2)$$
such that
\begin{multline*}H(z,w)=\Gamma(c^2)+\<\Omega(c^2),z' {w'}-c^2\>+\\ \bigl\<z'{w'}-c^2),F(z',w',c^2)(z'{w'}-c^2)\bigr\>,\end{multline*}
where
$$z'=z+\pa_{w'}f(z,w',c^2),\quad w=w'+\pa_{z}f(z,w',c^2),$$
then  $\G(c^2)=N(c^2)$ and $\Omega(c^2)=\nabla N(c^2)$, i.e. the series $\G$ and $\Omega$ are unique.
\end{Prop}

Proposition \ref{bnfIIelliptic-zw}  will be the consequence of the following two Lemmata.

\begin{Lem}\label{reduction}Let $H(z,w,c^2)$ be a formal Hamiltonian depending on $c^2$  of the form 
$$H(z,w,c^2)=\Gamma(c^2)+\<\Omega(c^2),zw-c^2\>+\cO^2(zw-c^2)$$
 where $\Omega(0)=\om_0 $ is rationally independent.
Then there exists  
{ a  formal  series 
$$f(z,w',c^2)\in \bC[[z,w',c^2]]\cap \OO^2(zw'-c^2)$$
}
and formal series $G_{k}(c^2)\in\bC[[c]]$, $k\geq 2$, such that
\begin{multline*}H(z,w,c^2)=\Gamma(c^2)+\<\Omega(c^2),(z'w'-c^2)\>+\\ \sum_{n=2}^\infty G_{k}(c^2)\cdot (z'w'-c^2)^{\otimes n},
\end{multline*}
where
$$z'=z+\pa_{w'}f(z,w',c^2),\quad w=w'+\pa_{z}f(z,w',c^2).$$
\end{Lem}
\begin{proof}

Let us write $$H(z,w,c^2)=\Gamma(c^2)+\<\Omega(c^2),zw-c^2\>+F(z,w,c^2)\cdot (zw-c^2)^{\otimes 2}$$
and  denote $\Delta=zw-c^2$, $\ti\Delta=zw'-c^2$ and $\Delta'=z'w'-c^2$.
We have
\begin{align*}\Delta&=zw'+z\pa_{z} f(z,w',c^2)-c^2\\
&=\ti\Delta+z\pa_{z}f(z,w',c^2)
\end{align*}
and 
\begin{align*}\Delta'&=zw'+w'\pa_{w'} f(z,w',c^2)-c^2\\
&=\ti\Delta+w'\pa_{w'}f(z,w',c^2)
\end{align*}
We thus  have to construct $f=\cO^3(z,w',c^2)$ and the $G_{k}$ such that 
\begin{multline}\<\Omega(c^2),\cD f(z,w',c^2))\>=\\- F (z,w'+\pa_{z}f(z,w',c^2),c^2)\cdot(\ti\Delta+z\pa_{z}f(z,w',c^2))^{\otimes 2}+\\
 \sum_{n=2}^\infty G_{n}(c^2)\cdot (\ti\Delta+w'\pa_{w'}f(z,w',c^2)^{\otimes n}.\label{eq1}
\end{multline}
To do this we proceed by induction. If $L$ is a formal function of the variables $(z,w',c^2)$ denote by $[L]_{j}$ its  homogeneous part of total degree $j$ in { $(z,w',c)$}. 
Taking the $[\cdot]_{j}$ part of equation (\ref{eq1}) we get 
\begin{multline}\sum_{k+l=j}\<[\Omega(c^2)]_{k},[\cD f(z,w',c^2))]_{l}\>=\\-[F (z,w'+\pa_{z}f(z,w',c^2),c^2)\cdot(\ti\Delta+z\pa_{z} f(z,w',c^2))^{\otimes 2}]_{j}+\\
 \sum_{n=2}^\infty [G_{n}(c^2)\cdot (\ti\Delta+w'\pa_{w'}f(z,w',c^2))^{\otimes n}]_{j}.\label{eq2}
\end{multline}

{ Since $[f]_{0}=[f]_{1}=[f]_{2}=0$,  (\ref{eq2}) can} be written 
 \begin{multline} \<\Omega(0),[\cD f(z,w',c^2))]_{j}\>=T_{1,j}+T_{2,j}+T_{3,j} +\\ \sum_{n\ge s2} [G_{n}(c^2)]_{j-2n}\cdot (\ti\Delta)^{\otimes n} \label{eq2'}\end{multline} 
 where 
 \be T_{1,j}=-\sum_{\ntop{k+l=j}{{ 2\leq l < j}}}\<[\Omega(c^2)]_{k},[\cD f(z,w',c^2))]_{l}\>\ee
 \be T_{2,j}= -[F (z,w'+\pa_{z}f(z,w',c^2),c^2)\cdot(\ti\Delta+z\pa_{z} f(z,w',c^2))^{\otimes 2}]_{j} \ee
 \be T_{3,j}= \sum_{n=2}^\infty \sum_{k+l=j}[G_{n}(c^2)]_{k}\cdot [(\ti\Delta+w'\pa_{w'}f(z,w',c^2))^{\otimes n}]_{l}- \sum_{n=2}^\infty [G_{n}(c^2)]_{j-2n}\cdot (\ti\Delta)^{\otimes n}\ee

Notice that the term $T_{2,j}:=[F (z,w'+\pa_{z}f(z,w',c^2),c^2)\cdot(\ti\Delta+z\pa_{z} f(z,w',c^2))^{\otimes 2}]_{j}$  is a linear combination of terms of the form $[F(z,w'+\pa_{z}f(z,w',c^2),c^2)]_{k}\cdot \ti\Delta^{\otimes m_{1}}\otimes_{sym}[z\pa_{z}f(z,w',c^2)]_{l_{2}}^{\otimes m_{2}}$ with $k+2m_{1}+m_{2}l_{2}=j$ and $m_{1}+m_{2}=2$; hence $l_{2}\leq j-1$. Also, since $[z\pa_{z}f]_{0}=0$ one has $l_{2}\geq 1$ and thus $k\leq j-2$. The term  $[F(z,w'+\pa_{z}f(z,w',c^2),c^2)]_{k}$ depends on $f$ only through its coefficients of total degree   $\leq k+1$ and thus less or equal to $j-1$. In conclusion the term $T_{2,j}$ depends on $f$ only through its coefficients of total degree $k\leq j-1$. 

A similar analysis  shows that the same is true for the term $T_{3,j}$ and of course,  for the term $T_{1,j}$.  In conclusion, all the terms $T_{1,j},T_{2,j},T_{3,j}$ in equality (\ref{eq2'}), except $\<\Omega(0),[\cD f(z,w',c^2))]_{j}\>$ depend on $f$ only through its coefficients  of total degree less or equal to $j-1$. 

{ Moreover, by assumption on $f$, the derivatives
$z\pa_{z} f(z,w',c^2)$ and $w'\pa_{w'} f(z,w',c^2)\in \OO^1(\tilde \De)$ and, hence,  $T_{2,j}$
and $T_{3,j} \in \OO^2(\tilde \De)$. Since $\cD f(z,w',c^2))\in \OO^2(\tilde \De)$, also $T_{1,j} \in \OO^2(\tilde \De)$.   }

Finally, since  $[w'\pa_{w'}f]_{j}=0$ for $j=0,1,2$ we see that the term $T_{3,j}$ depends on the $[G_{n}]_{l}$ only for $l<j-2n$.

We can now construct by induction $[f]_{j}$ and the $[G_{n}(c^2)]_{j-2n}$ for all the $n$ such that  $2n\leq j$. For $j=3$ it is enough to choose $[f]_{3}=0$. Then assuming we have constructed $[f]_{k}$ and $[G_{n}(c^2)]_{k}$ for all $k\leq j-2n$, $3\leq k\leq j-1$, we can find $[f]_{j}$ and $[G_{n}]_{j-2n}$ such that  (\ref{eq2'}) holds: indeed, 
{ we define 
$$\sum_{n\ge 2} [G_{n}(c^2)]_{j-2n}\cdot (\ti\Delta)^{\otimes n}=-\cM(T_{1,j}+T_{2,j}+T_{3,j})$$
}
and we apply item \ref{item4} of Lemma \ref{cohom} with $\omega=\Omega(0)$.
\end{proof}

The second statement is about uniqueness. 
\begin{Lem}\label{bnfIelliptic2}
Assume that there exist  { a formal  series
$$f(z,w',c^2)\in \bC[[z,w',c^2]]\cap \OO^2(z,w',c)$$
}
and formal series $\Gamma(c^2),\Omega(c^2), F_{k}(c^2) \in\bC[[c]]$  such that  
\begin{multline*}N(zw)=\Gamma(c^2)+\<\Omega(c^2),z'{w'}-c^2\>+\sum_{n=2}^\infty F_{n}(c^2)\cdot (z'{w'}-c^2)^{\otimes n},\end{multline*}
where
$$z'=z+\pa_{w'}f(z,w',c^2),\quad w=w'+\pa_{z}f(z,w',c^2),$$
then  $\G(c^2)=N(c^2)$ and $\Omega(c^2)=\nabla N(c^2)$ (thus they are unique).
\end{Lem}
\begin{proof}
Let us denote $\ti\Delta=zw'-c^2$ and $\Delta'=z'w'-c^2$.
By assumption
$$N(zw'+z\pa_{z}f(z,w',c^2))=\Gamma(c^2)+\<\Omega(c^2),\Delta'\>+\sum_{n=2}^\infty F_{n}(c^2)\cdot (\Delta')^{\otimes n}$$
and using the fact that 
\begin{align*}zw'+z\pa_{z}f(z,w',c^2)&=zw'+w'\pa f(z,w',c^2)+\cD f(z,w',c^2)\\
&=c^2+\ti\Delta+\cD f(z,w',c^2)+w'\pa_{w'}f(z,w',c^2)
\end{align*}
and 
\be \Delta'
=\ti\Delta+w'\pa_{w'}f(z,w',c^2)
\ee
we have
\begin{multline*} N(c^2+\ti\Delta+\cD f(z,w',c^2)+w'\pa_{w'}f(z,w',c^2))=\\ \Gamma(c^2)+\<\Omega(c^2),\ti\Delta+w'\pa_{w'}f(z,w',c^2)\>+\sum_{n=2}^\infty F_{n}(c^2)\cdot (\ti\Delta+w'\pa_{w'}f(z,w',c^2))^{\otimes n}
\end{multline*}
Using Taylor formula
\begin{multline*}N(c^2)+\nabla N(c^2)\cdot (\ti\Delta+\cD f(z,w',c^2)+w'\pa_{w'}f(z,w',c^2))+\\ \sum_{|m|\geq 2}\pa^mN(c^2)\cdot (\ti\Delta+\cD f(z,w',c^2)+w'\pa_{w'}f(z,w',c^2))^{\otimes m}=\\ \Gamma(c^2)+\<\Omega(c^2),\ti\Delta+w'\pa_{w'}f(z,w',c^2)\>+\sum_{n=2}^\infty F_{n}(c^2)\cdot (\ti\Delta+w'\pa_{w'}f(z,w',c^2))^{\otimes n}\end{multline*}
so
\begin{multline}\nabla N(c^2)\cdot \cD f(z,w',c^2)+\sum_{|m|\geq 2}\pa^mN(c^2)\cdot (\ti\Delta+\cD f(z,w',c^2)+w'\pa_{w'}f(z,w',c^2))^{\otimes m}=\\ (\Gamma(c^2)-N(c^2))+\<\Omega(c^2)-\nabla N(c^2),\ti\Delta+w'\pa_{w'}f(z,w',c^2)\>+\\ \sum_{n=2}^\infty F_{n}(c^2)\cdot (\ti\Delta+w'\pa_{w'}f(z,w',c^2)))^{\otimes n}. \label{eq5}   \end{multline}

We denote by $[\cdot]_{j}$ the homogeneous polynomial part of total degree $j$ in 
the  { $(z,w',c)$} variables. 
{Using that $[f]_0=[f]_1=0$, it follows readily from  equation \eqref{eq5} that
$$\G(c^2)=\<\om,c^2\>+\OO^2(c^2),\quad \Omega(c^2)=\omega+\OO(c^2).$$
}

We now prove  that $\cD f(z,w',c^2)=0$. 
 We shall prove by induction on $j$ that for any $j\geq 0$,  $\cD [f]_{j}=0$. By assumption 
 { this is true for $j=0,1$. Let us assume this is true for all $1\leq k\leq j-1$.} By taking the $[\cdot]_{j}$ in equation (\ref{eq5}) and using items \ref{item2} and \ref{item3} of Lemma \ref{cohom} ($[\cD f]_{k}=\cD [f]_{k}$,  $[w'\pa_{w'}f]_{k}=w'\pa_{w'}[f]_{k}$), the fact that $[f]_{l}=0$ { for $0\leq l \leq 1$} and $[\Omega(c^2)-\nabla N(c^2)]_{0}=0$  we get 
\begin{multline} \nabla N(0)\cdot \cD [f]_{j}(z,w',c^2)=[(\Gamma(c^2)-N(c^2))]_{j}+\<[\Omega(c^2)-\nabla N(c^2)]_{j-2},\ti\Delta\>+\\ S_{1,j}+S_{2,j}+S_{3,j} +S_{4,j} \label{eq6}\end{multline}
where 
\be S_{1,j}=-\sum_{\ntop{k+l=j}{l\geq 2,k\geq 1}}[\nabla N(c^2)]_{k}\cdot \cD [f]_{l}(z,w',c^2) \ee
\be S_{2,j}=-[ \sum_{|m|\geq 2}\pa^mN(c^2)\cdot (\ti\Delta+\cD f(z,w',c^2)+w'\pa_{w'}f(z,w',c^2))^{m}]_{j} \ee 
\be S_{3,j}= \sum_{\ntop{k+l=j}{l\geq 2,k\geq 1}}\<[\Omega(c^2)-\nabla N(c^2)]_{k},w'\pa_{w'}[f]_{l}(z,w',c^2)\> \ee
\be S_{4,j}= \sum_{n=2}^\infty [F_{n}(c^2)\cdot (\ti\Delta+w'\pa_{w'}f(z,w',c^2)))^{\otimes n}]_{j}. \ee
We also  observe that the sum $S_{2,j}:= [ \sum_{|m|\geq 2}\pa^mN(c^2)\cdot (\ti\Delta+\cD f(z,w',c^2)+w'\pa_{w'}f(z,w',c^2))^{m}]_{j}$ is a linear combination of multilinear terms of the form $[\pa^mN(c^2)]_{l_{1}}\cdot \ti\Delta^{\otimes_{sym} m_{1}}\otimes_{sym}  [\cD f(z,w',c^2)]_{l_{2}}^{\otimes m_{2}}\otimes_{sym} [w'\pa_{w'}f(z,w',c^2)]_{l_{3}}^{\otimes m_{3}}$ with $l_{1}+2m_{1}+m_{2}l_{2}+m_{3}l_{3} =j$ and $m_{1}+m_{2}+m_{3}=m\geq 2$; hence $\max(m_{2},m_{3})\leq j-1$.   But the induction assumption implies that $[f]_{m}$ is diagonal for $m\leq j-1$, hence the same is true for  $\cD [f]_{m}$ and $w'\pa_{w'}[f]_{m}$. In conclusion, the sum $S_{2,j}$ is diagonal. A similar argument shows that the sum $S_{4,j}:=\sum_{n=2}^\infty [F_{n}(c^2)\cdot (\ti\Delta+w'\pa_{w'}f(z,w',c^2)))^{\otimes n}]_{j}$ is diagonal, as well as all the other terms of the equation  (\ref{eq6}) with  the possible exception of the term  $\<\nabla N(0),\cD [f]_{j}\>$. It then follows that $\<\nabla N(0),\cD [f]_{j}\>$ is diagonal. By Lemma \ref{cohom} this forces $\cD [f]_{j}=0$. 
This completes the induction and proves that $\cD f=0$.
Now equation (\ref{eq5}) reads
\begin{multline}(\Gamma(c^2)-N(c^2))+\<\Omega(c^2)-\nabla N(c^2),\Delta'\>=\\ =\sum_{|m|\geq 2}\pa^mN(c^2)\cdot (\Delta')^{\otimes m}- \sum_{n=2}^\infty F_{n}(c^2)\cdot (\Delta')^{\otimes n}. \label{eq7}   \end{multline}
hence
\be (\Gamma(c^2)-N(c^2))+\<\Omega(c^2)-\nabla N(c^2),\Delta'\>=\cO^2(\Delta') \ee
and Lemma \ref{uniquedecomp} concludes the proof.
\end{proof}

\begin{proof} 
{We can now prove Proposition  \ref{bnfIIelliptic-zw}. Using Lemma \ref{formcomp} we can assume that $H$ is under Birkhoff Normal Form and we then apply consecutively Lemmas \ref{reduction} and \ref{bnfIelliptic2}.

Indeed, using Lemma  \ref{formcomp}(c)  there is, by assumption, a (formally) exact mapping $Z_1$ in 
$(*)$ such that
\begin{multline*}H\circ Z_1(z,w,c^2)=\Gamma(c^2)+\<\Omega(c^2),zw-c^2\>+\\ + \bigl\<zw-c^2),F(z,w,c^2)(zw-c^2)\bigr\>,\end{multline*}
Using  Lemmas \ref{reduction} and \ref{formcomp}(c) there is a (formally) exact mapping $Z_2$ in 
$(*)$ such that
\begin{multline*}H\circ Z_1\circ Z_2(z,w,c^2)=
\Gamma(c^2)+\<\Omega(c^2),(z'w'-c^2)\>+\\ +\sum_{n=2}^\infty G_{k}(c^2)\cdot (z'w'-c^2)^{\otimes n}.
\end{multline*}
By the Birkhoff normal form there is a (formally) exact mapping $Z_3$ in 
$(*)$ such that
$$H\circ Z_3(z,w)=N(zw)= \<\om_0,zw\>+\OO^3(zw).$$
Hence $W=Z_3^{-1}\circ Z_1\circ Z_2$ is (formally) exact in $(*)$, by Lemma  \ref{formcomp}(a), and 
\begin{multline*}N\circ W (z,w,c^2)=
\Gamma(c^2)+\<\Omega(c^2),(z'w'-c^2)\>+\\ +\sum_{n=2}^\infty G_{k}(c^2)\cdot (z'w'-c^2)^{\otimes n}.
\end{multline*}
By Lemma  \ref{formcomp}(b), $W$ has a generating function and now the proposition follows from
Lemma \ref{bnfIelliptic2}.
}
 \end{proof}

\section{A KAM counter term theorem and the {frequency map}}\label{s3}
The proof of Theorem \ref{theo:1.1} relies on the following fundamental result:
\begin{Prop}\label{counterthm}
Given $0<\k<1$ and $\t>d-1$. Then, for all $s\in\N$, there exist 
{ non-negative
constants   (only depending on $s$ and $\t$) 
$$\al(s)\ge (s-t)+\al(t),\quad s\ge t\ge0,$$
}
such that if 
$$H(z,w)= N^q(zw)+\OO^{2q+1}(z,w)\in\CC^{\w}(\C^{2d},0),\quad q\ge \al(1)+1,$$
is $\sigma$-symmetric
with 
$$N^q(r)=\l{\o_0,r}+\OO^2(r),$$
{ then there exist $\de>0$ and for any $\eta<\de$ a $\sigma$-symmetric function
$$\Lambda=\Lambda(c,\o)\in \CC^{\w,\i}(\bD^{d}_{\eta}\times B)$$
and a symplectic and $\sigma$-symmetric diffeomorphism
$$
(Z_{c,\om}-\mathrm{id})(z,w)\in \CC^{\w,\i}(\bD^{2d}_\eta\times \bD^{d}_{\eta}\times B)\cap \OO^{2}(z,w,c)$$
such that
\begin{multline} \label{p35}
(H+\l{\o+\Lambda(c,\o), \cdot}) \circ Z_{c,\om}(z, w)\\
=\l{\o,zw-c}+\OO^2(zw-c) +g(z,w,c,\omega)\end{multline}
}
(modulo an additive constant that depends on $c,\o$) with $g$ $(\k,\t)${\it -flat} and 
{  $g\in \cO^2(z,w,c)$.}

Moreover,
\begin{itemize}
\item[(i)]
for any $s\in\N$ 
there exists a constant $C_s$, only depending on $s,H,\t$ such that
$$\aa{\Lambda+\p_rN^q}_{\eta,s}+\aa{Z-\mathrm{id}}_{\eta,s}\le C_s\eta^{q}
(\frac{1}{\k\eta})^{\al(s)}$$

\item[(ii)] there exists a constant $C$, only depending on $H,\t$,
such that
$$\de \ge \frac1C\k^{\frac{\al(1)}{q-\al(1)}}$$
\item[(iii)] if
$$\o_0\in DC(2\k,\t)$$
 then the mapping
$$\bD^{d+1}_{\de'}\ni (c,\la)\mapsto \Lambda(c,(1+\la)\o_0)\in \bC^d$$
is  holomorphic and $\s$-symmetric  for some $0<\de'<\delta$

{
\item[(iv)]
$$(Z_{c^2,\om}-\mathrm{id})(z,w)\in\cO^{2q}(z,w,c).$$
}
\end{itemize}

\end{Prop}

\begin{Rem}
Notice that this proposition (except part (iii)) does not require that $\o_0$ is Diophantine. {
Notice also that, a priori, $\Lambda$, $Z$ and $g$ depend on $\kappa$ and on $\eta$.}
\end{Rem}

\begin{Rem}
{ It is also the case that  $g(z,w,c^2,\omega)$ and  $\Lambda(c^2,\omega)+\nabla N^q(c^2)$ are in $\cO^{2q}(z,w,c)$, but we shall not use this fact. }
\end{Rem}

This proposition follows from   the local Normal Form Theorem  \ref{theo:5.1}  applied to the Hamiltonian $\ti H(z,w,c)=H(z,w)-N^{q}(c)-\<\nabla N^{q}(c),zw-c\>$ in a similar way as Proposition 4.2 of \cite{EFK}.
Let us discuss this a bit, but for full details we refer to \cite{EFK}.

If we write $F(z,w)=H(z,w)-N^q(zw)=\cO^{2q+1}(z,w)$, then (with the notation of Section \ref{sec:2}) 
\begin{multline*} H(z,w)=N^q(c^2)+\<\nabla N^q(c^2),zw-c^2\>+\cO^2(zw-c^2)+\\
F^{(0)}(z,w,c^2)+\<F^{(1)}(z,w,c^2),zw-c^2\>+F^{[2]}(z,w,c^2)\cdot(zw-c^2)^{\otimes 2}
\end{multline*}
where $F^{(0)}(z,w,c^2)=\cO^{2q+1}(z,w,c)$ and $F^{(1)}(z,w,c^2)=\cO^{2q-1}(z,w,c)$.
{ Hence
\begin{multline}  \ti H(z,w,c)=F^{(0)}(z,w,c)+\\ +\<F^{(1)}(z,w,c),zw-c^2\>+\cO^2(zw-c).
\end{multline}
On domains where $\max(|z|,|w|,|c|)< \eta$,
${\ti H}^{(0)}(z,w,c,\omega)$ is of order $\eta^{q+1}$ and 
${\ti H}^{(1)}(z,w,c,\omega)$ is of order $\eta^{q}$. 
Using Lemma \ref{mainlemma:elliptic} we obtain 
that $[\tilde H]_{\eta,0}$ is of order $\eta^{q-b}$, where $b$ is a constant only depending on $\tau$ and $d$. 
If we take $h$ equal $\frac\eta4$, say,  then the smallness assumption \eqref{smallness} is fulfilled for any
$$\ \eta\le \frac1C\k^{\frac{a}{q-a-b}}$$
--  this gives the estimate of $\de$ in (ii).

If we  call $\ti \Lambda(c,\omega)$ the counter term $\Lambda$ obtained by applying Theorem \ref{theo:5.1} to $\ti H(z,w,c)$, we then get the conjugacy equation (\ref{p35}) with $\Lambda(c,\omega):=\ti\Lambda(c,\omega)-\nabla N^q(c)$. Since $\ti \Lambda$ is small (as quantified \eqref{upperbound} in Theorem \ref{theo:5.1}) we get also the first half of the  inequality given in item(i). The second half of the inequality in(i) also follows from  \eqref{upperbound}.

Item(iii) follows from the last part of Theorem  \ref{theo:5.1}.  

Since $\ti H(z,w,c)=\cO^{2q+1}(z,w,c)$ mod $\cO^2(zw-c)$ we have(iv) by Theorem \ref{theo:5.1}. }

\begin{Cor}\label{bnfIII} Given $0<\k<1$ and $\t>d-1$ and 
 non-negative constants  $\al(s)$
as in Proposition \ref{counterthm}, if 
$$H(z,w)= N^q(zw)+\OO^{2q+1}(z,w)\in\CC^{\w}(\{0\}),\quad q\ge \al(1)+1,$$
with 
$$N^q(r)=\l{\o_0,r}+\OO^2(r),$$
then, 
{ for any
$$\eta < \eta_0=\frac1{C'}\k^{\frac{\al(1)}{q-\al(1)}},$$
there exists a unique $\CC^{\i}$ function $\Omega:\{ c\in \R^d: |c|< \frac\eta2\} \to \R^d$
such that 
$$\Omega(c)+\Lambda(c,\Omega(c))=0,\quad \forall c$$
}
Moreover,
\begin{itemize}

\item[(i)]  for any $s\in\bN$ there exists a constant $C'_s$, only depending on $s,H,\t$ such that
$$\aa{\Omega-\p_rN^q}_{\CC^s(|c|<\frac\eta2)}\le  
C'_s\eta^q(\frac{1}{\k\eta})^{\al(s)}$$

\item[(ii)] If $\o_0\in {\rm DC}(\tau,\kappa)$, the Taylor series of $\Omega$ at $c=0$ is given by
$\nabla N_H(c)$. 

\end{itemize}

The constants  $C'_s$ only depend on $H,\t$.
\end{Cor}

We call $\Omega$ the {\it frequency map}. The proof of the corollary is almost identical to the one of Corollary 4.3 of \cite{EFK}. 
Let us therefore only discuss shortly the proof.

The existence of $\Omega$ and the estimate (i) follow from (i) of Proposition  \ref{counterthm} and the implicit function theorem applied to the function $\Lambda$. 

{Point (ii) is a  consequence of  the following facts:  if $(z,w)\mapsto Z_{c}(z,w):=Z_{c,\Omega(c)} (z,w)$ is the change of variable given in Proposition \ref{counterthm}, then
\be H\circ Z_{c^2}(z,w)=\<\Omega(c^2),zw-c^2\> +\cO^2(zw-c^2)+g(z,w,c^2,\Omega(c^2)).\ee
The condition $\o_0 \in {\rm DC}(\tau,\kappa)$  and the fact that $g$ is $(\kappa,\tau)$-flat show that $g(z,w,c^2,\Omega(c^2))=\cO^\infty(z,w,c^2)$ 
and hence, one has  in $\bC[[z,w,c]]$
\be H\circ Z_{c^2}(z,w)=\<\Omega(c^2),zw-c^2\> +\cO^2(zw-c^2).\ee
{ Since $(Z_{c^2}-id)(z,w)=\cO^{2}(z,w,c^2)$, $Z_{c^2}$} has  a formal generating function 
$f(z,w,c^2)\in \cO^{3}(z,w,c^2)$.
Proposition \ref{bnfIIelliptic-zw} then shows that in $\bC[[z,w,c]]$ one has the identity $\Omega(c^2)=\nabla N(c^2)$. 

\medskip

We shall use the preceding results to prove Theorems \ref{mA}--\ref{mB}.
In the case the BNF is non degenerate, $q$ is chosen according to the non-degeneracy condition, and it then will follow from(i) 
 that the function $\Omega(c)$  --  which depends on $\kappa$ --  takes values 
in $ {\rm DC}(\tau,\kappa)$ on a set of positive measure which insures KAM stability. This
will be proven in Sections \ref{s35} and \ref{s36}.
In the case of a degenerate BNF,(ii) of Proposition \ref{counterthm} as well as the analyticity of $\Lambda(c,\o)$ in the variable $c$ allows to conclude the proof of Theorem \ref{mA} and \ref{Russmann}. This will be carried out  in Section \ref{sec:mA}
}

\section{Proof of the main results}

This section is devoted to the derivation of Theorems  \ref{mA}--\ref{mB}, and thus of  Theorem \ref{theo:1.1}, from Proposition  \ref{counterthm} and Corollary \ref{bnfIII}.

{  Consider a real analytic Hamiltonian $H$ of the form $(*)$. By a real symplectic conjugation we can
assume, since $\om_0$ is rationally independent,  that $H$ is on Birkhoff normal form up to order $2q+1$
for any $q$:
$$H(x,y)=N^q(\frac12(x^2+y^2))+\OO^{2q+1}(x,y).$$
Performing the linear change of variable in Section \ref{strategy} we obtain
a $\sigma$-symmetric holomorphic Hamiltonian
$$\ti H(z,w)=N^q(zw)+\OO^{2q+1}(z,w)$$
of the form treated in Proposition  \ref{counterthm} and Corollary \ref{bnfIII}.

The linear change of variable is not symplectic and it will change the canonical symplectic structure $dx\wedge dy$ into $\sqrt{-1}dz\wedge dw$. However, any transformation symplectic with respect to 
$dz\wedge dw$ will also be symplectic with respect to $\sqrt{-1}dz\wedge dw$, so we may just as well
study $\ti H$ under a transformation symplectic with respect to 
$dz\wedge dw$. Then the Birkhoff normal forms $N_H$ and $N_{\ti H}$ are the same and coincide with $N^q$
up to order $q$.
}

\subsection{Transversality} \label{s35}
Let us state two lemmas the proof of which can be found in Section 5 of \cite{EFK}.
\begin{Lem}\label{transversality}
If $N_H(r)$ is  non-degenerate, then
there exist $p,\s>0 $ such that
for any $k\in\Z^d \sm \{0\}$ there exists a unit vector
{  $u_k\in(\bR_+)^d$} such that the series 

$$
f_k(r)=\l{\frac{k}{|k|},\p_r N_H(r)}$$
is $(p,\s)${\it-transverse in direction} $u_k$, i.e.
$$\max_{0\le j\le p}
|\p^j_{t} f_k(tu_k)_{|t=0}|\ge \s.$$
\end{Lem}
Consider now these $p,\s$.
Let $\Omega\in \CC^p(\b{|c|<\eta})$
and assume
$$\aa{\Omega-[\p_r N_H]^p}_{\CC^p(\b{|c|<\eta})}\le \frac\s2.$$

\begin{Lem}\label{pyartly}
If $N_H$ is $(p,\s)$-transverse (in some direction), then
$$\Leb\b{|c|<\eta: |\l{\frac{k}{|k|},\Omega(c)}|<\ep}
\le C_p(\frac\ep\s)^{\frac1p}\eta^{d-1}$$
for any { $\eta, k,\eps$}.
\end{Lem}

\subsection{Proof of Theorem \ref{mB}}\label{s36} 

By Lemma \ref{transversality} we are given $p$ and $\s$ that correspond 
to the transversality of the Birkhoff normal form $N_H=N_{\ti H}$. 
We can assume without restriction that $\s\le 1$, and we 
fix $q=(1+2p)\a(p)+1$. 

We shall apply (i) of Proposition \ref{counterthm} and Corollary \ref{bnfIII} to $\ti H$
with  this $q$ and with

$$\t=dp+1 \quad \textrm{and}\quad 0<\k\le\s^q\le 1.$$
{Now let 
\begin{equation} \label{etaa} \eta:=\frac1{C''}(\frac\k\s)^{\frac1{2p}}. \end{equation}
Since $q\ge (1+2p)\al(1)+1$ we have $\eta\le \eta_0$ 
for all $C'' \geq C'$, where $\eta_0$ and $C'$ are defined in Corollary \ref{bnfIII}.  
Then $\Omega=\Omega_\kappa$
\footnote{\ not to forget that $\Omega$ depends on $\kappa$}
 is defined in $\b{|c|<\frac\eta2}$ and
\begin{equation}\aa{\Omega-[\p_r N_H]^p}_{\CC^p(\b{|c|<\frac\eta2})}\le C'_p\eta^q(\frac{1}{\k\eta})^{\al(p)}+
\aa{[\p_r N_H]^p-\p_r N_H^q}_{\CC^p(\b{|c|<\frac\eta2})}
\end{equation}
which is $$\leq \tilde{C}\eta$$ 
since $q\ge (1+2p)\al(p)+1$ --  notice that $\tilde C$ is independent of $C'' \geq C'$.
Finally if $C''$ is sufficiently large (depending on $p,\t$,$H$, thus on $q$) 
we have that $\tilde{C}\eta\leq \s/2$.
}

By Lemma \ref{pyartly}
$$ \Leb\b{|c|<\frac\eta2: |\l{\frac{k}{|k|},\Omega(c)}|<\ep} \lsim (\frac\ep\s)^{\frac1p}\eta^{d-1},$$
hence 
\begin{align*} \Leb\b{|c|<\frac\eta2: \Omega(c)\notin DC(\k,\t)}&\lsim (\frac{\k}{\s})^{\frac1p}\eta^{d-1}\\
&\lsim \eta\Leb\b{|c|<\frac\eta2} \end{align*}
(provided $\kappa$ is sufficiently small). Hence, the set
{
$$\Sigma_\kappa= \b{|c|<\frac\eta2: \Omega(c)\in DC(\k,\t)}\cap  \bR_+^d$$
}
is of positive measure when $\kappa$ is sufficiently small and has density $1$ at $0$
when $\kappa\to 0$. For each $c\in \Sigma_\kappa$
$$\ti Z_{c,\Omega(c),\kappa}(\{ zw=c\})
$$
is an invariant set for the Hamiltonian system defined by $\ti H$ with respect to the the canonical
symplectic structure  $dz\wedge dw$, hence also with respect to the symplectic structure
 $\sqrt{-1}dz\wedge dw$. 
 
Returning to the variables $x,y$, using the linear transformation defined in 
Section \ref{strategy}, we get  for 
any $c\in \Sigma_{\kappa}$ a symplectic transformation $Z_{c,\kappa}: (\bR^{2d},0)\to  (\bR^{2d},0)$ 
\footnote{ \ $Z_{c,\kappa}$ is real because $\ti Z_{c,\Omega(c),\kappa}(z,w)$ is $\sigma$-symmetric}
such that 
$$Z_{c,\kappa}(\{ x^2+y^2=c\})$$
is a KAM-torus for the Hamiltonian system defined by $H$. By (iv) of Proposition
\ref{counterthm} $Z_{c,\kappa}$ has the form
$$ Z_{c,\kappa}(x,y)=(x,y)+\OO^{q}(x,y,c).$$

Let now $W_\kappa: (\bR_+^d,0)\times \bT^d\to  (\bR^d\times \bR^d,0)$ be the mapping
$$(c,\theta)\mapsto Z_{c,\kappa}(c\cos2\pi\theta,c\sin2\pi\theta)= (c\cos2\pi\theta,c\sin2\pi\theta)+\OO^q(c).$$ 
Then 
$$W_\kappa(\{c:c \in \Sigma_{\kappa}\},\bT^d)$$
is foliated into KAM-tori. {By \eqref{etaa}    and the estimate (i) of Proposition \ref{counterthm}    we have that the $\cO^q(c)$ term in $W_\k$ satisfies the condition \eqref{jac} of Lemma \ref{lemmaA1} of the appendix A, which hence yields that $W_\kappa(\{c:c \in \Sigma_{\kappa}\},\bT^d)$ has positive measure when $\kappa$ is sufficiently small
and that the union over all $\kappa>0$ has density $1$ at  $0\in \bR^d\times\bR^d$.}

\subsection{Proof of Theorems \ref{mA} and \ref{Russmann} }\label{sec:mA}

We shall apply Proposition \ref{counterthm} and (ii) of Corollary \ref{bnfIII}
with $q=\a(1)+1$ and
$$q=\al(1)+1,\quad
\t=\t_0\quad \textrm{and}\quad \k=\frac{\k_0}2.$$

Then
$$\Omega(c)+\Lambda(c,\Omega(c))=0$$
and
$$\Omega(c)=\p_rN_H(c)+\OO^{\i}(c).$$

Since $N_H$ is $j$-degenerate
we have
$$\p_v^nN_H(0)=0\qquad \forall n\ge0$$
for any $v\in\Lin(\g=(\g_1,\dots,\g_j))$,
where $\p_v$ is the directional 
derivative in direction $v$. From this we derive that
$$\p_v^n(\o_0+\L(\cdot,\o_0))_{|c=0}=0\qquad \forall n\ge0.$$
By (iii) of Proposition \ref{counterthm} 
$s\mapsto \L(\l{s,\g},\o_0)$ is an analytic function in $s \in \R^j$, $s\sim 0$,
it must be identically $0$, hence 
$\Omega(\l{s,\g})$ is identically $\o_0$,
i.e.
$$\Omega(\l{s,\g})=\om_0 \in DC(\k_0,\t)\sbs DC(\k,\t)$$
for all  $|s| \le s_+$.
{ Thus we have
\be\label{5.32}
H\circ \ti Z_{c,\om_0}(z,w)=\<\om_0,zw-c\> +\cO^2(zw-c)+g(z,w,c,\om_0)\ee
for any
$$c\in  \Delta=\{c=\l{s,\g}: |s| \le s_+ \}.$$
Since everything is analytic in $s$, \eqref{5.32} extends to complex $s$ in some neighborhood of $0$.

Hence, for any $c\in  \Delta$ the set
$$\ti Z_{c,\om_0}(\{ zw=c\})$$
is an invariant Lagrangian submanifold for the Hamiltonian system defined by by $\ti H$ with respect to the the canonical
symplectic structure  $dz\wedge dw$, hence also with respect to the symplectic structure
 $\sqrt{-1}dz\wedge dw$. 

The set
$$\bigcup_{c\in\Delta}   \{ zw=c\} \times\{c\}\sbs \bC^{2d}\times \bC^d $$
is an analytic submanifold of (complex) dimension $d+j$, singular at the origin. It's image $M$
under the holomorphic diffeomorphism
$$(z,w,c)\mapsto (\ti Z_{c,\om_0}(z,w),c)$$
is therefore an analytic submanifold of (complex) dimension $d+j$, singular at the origin. 
The image of M under the projection on $ \{ z,w\}$ is a subanalytic set.

Using (iv) of Theorem \ref{counterthm}, it is easy to find points on (any component of) $M\ni 0$ where this projection, 
restricted to $M$, is onto. The image of $M$ under the projection is therefore an analytic
subvariety of (complex) dimension $d+j$.  This completes the proof of Theorem \ref{mA}.
}

When $N_H$ is $(d-1)$-degenerate, then
$$\p_r N_H(c)=\m(\l{c,\o_0})\o_0$$
where $\m(t)=1+\OO(t)$ is a formal power series in one variable.

Since 
$$\m(\l{c,\o_0}) \o_0+\Lambda(c,\m(\l{c,\o_0})\o_0)=\OO^{\i}(c),$$
taking $c=t\o_0$, we have (assuming $\o_0$ is a unit vector) 
\be\label{3.6}\m(t)\o_0+\Lambda(t\o_0,\m(t)\o_0)=0\ee
modulo a term in $\OO^{\i}(t)$.
Since, by Proposition \ref{counterthm} (iii), 
the lefthand side is analytic in $t\o_0$ and
$\mu$ we obtain from (\ref{3.6})  that
$\mu(t)$ is a convergent power series. Then
$$t\mapsto \m(t)\o_0+\Lambda(t\o_0,\m(t)\o_0)$$
is analytic for $t\sim 0$, hence identically zero.
We derive from this that
$$\Omega(c)= \m(\l{c,\o_0})\o_0,$$
i.e.
$$\Omega(c)\in DC(\k,\t)$$
for all sufficiently small $c$. 
R\"ussmann's theorem now follows from an argument similar to that of the end of the proof of of Theorem \ref{mA}.
\bigskip

\section{The  (local) Normal Form Theorem}\label{sec6}

\subsection{Functional spaces and the operators $\PP$ and $\cL$}
We come back to the setting and notations of Section \ref{sec:2}, but we now consider the general case of functions $f(z,w,c,\omega)$ depending analytically on $z,w,c$ and smoothly on $\omega$.
Let $\delta>0$, and denote by $C^{\omega,\infty}_{\delta}$ the set of functions $f\in C^{\omega,\infty}(\bD^{2d}_{{ \delta}}\times \bD^d_{{\delta}}\times B)$ such that $f(z,w,c,\omega)\in \cO^2(z,w,c)$.

Let $\kappa,\tau$ be positive numbers and $l:\R\to\R$ a fixed even, non-negative $C^\infty$ function such that 
$|l| \leq 1$, and  $l(x)=0$ if $|x|\geq 1/2$ and 
$l(x) = 1$ if $|x|\leq 1/4$.  We  introduce the cut-off operator $\cP$: if  $f\in C^{\omega,\infty}(\bD^{2d}_{ \delta}\times \bD^{d}_{\delta}\times B)$,   $f(z,w,c,\omega)=\sum_{\a,\beta\in\N^d}f_{\a,\beta}(c,\omega)z^\a w^\beta$ then
$$\cP(f)(z,w,c,\omega)=\sum_{\a,\beta\in\N^d}f_{\a,\beta}(c,\omega)z^\a w^\beta l(\<\a-\beta,\omega\>\frac{(|\alpha|+|\beta|)^\tau}{\kappa}).$$
A function $f$ such that $\cP f=f$ satisfies $\pa_{z}^\a\pa_{w}^\beta\pa_{c}^\gamma\pa_{\omega}^\delta f(z,w,c,\omega)=0$ for any set of indices $\a,\beta,\gamma,\delta$ when $\omega\in DC(\kappa,\tau)$. In particular such a function is $(\kappa,\tau)$-{\it flat}.

Notice that $\PP$ and $\cM$ commute and that  $\PP$ preserves the space,   that we still denote $\widehat{\cN\cR}$, of maps $f(z,w,c,\omega)$ which for each fixed value of $\omega$ are in $\widehat{\cN \cR}$.

We now  define the linear  operator $\cL:C^{\omega,\infty}(\bD^{2d}_{{\delta}}\times \bD^d_{{ \delta}}\times B) \to C^{\omega,\infty}(\bD^{2d}_{{\delta'}}\times \bD^{d}_{{\delta}}\times B)$ ($\delta'<\delta$) by:  $\LL(f)=u$ if and only if 
\begin{equation}
\left\lbrace\begin{array}{l}
\cD^\omega u(z,w,c,\omega)=f(z,w,c,\omega)-\PP(f)(z,w,c,\omega)-\MM(f)(z,w,c,\omega)\\
\MM(u)=\PP(u)=0.
\end{array}\right.
\label{lineq}\end{equation}

Here is the analogue of Lemma 8.1 of \cite{EFK} the proof being the same (the only modification is to replace $d$ by $2d$).
\begin{Lem}\label{mainlemma:elliptic}One has
$$\max(  \|\cP(f)\|_{\delta',s},  \|\LL(f)\|_{\delta',s} )\leq  
C_{s}(\frac1\kappa)^{s+1}(\frac{1}{\delta-\delta'})^{(\tau+1)s+\tau+2d}
\|f\|_{\delta,s}$$
for any $\delta'<\delta$. The constant $C_s$ only depends, besides $s$, on $\t$ and $l$.
\end{Lem}
Since $\s$-symmetry is an important issue we mention  the following obvious lemma (see items \ref{item'4} and \ref{item4} of Lemma \ref{cohom}).
\begin{Lem}If $f\in \cC^{\omega,\sigma,\infty}_{\delta,s}$ then $\cP f$ and  $\sqrt{-1}\cL f$ are $\s$-symmetric. 
\end{Lem}

Let us also  mention the following fact:
\begin{Lem}\label{lem106}If $f\in C^{\omega,\infty}(\bD^{2d}_{{\delta}}\times \bD^d_{{ \delta}}\times B)$ then for $j=0,1,2,3$, $\delta'<\delta$,
$$\max(\|f^{(j)}\|_{\delta',s},\|f^{[j]}\|_{\delta',s})\leq C\frac{1}{(\delta-\delta')^{3d}}  \|f\|_{\delta,s}.$$
(The notations $f{^{(j)}}$, $f^{[j]}$ are defined in (\ref{candec1}), (\ref{candec2})).
\end{Lem}
\begin{proof}From (\ref{9candec}) and the fact that ($\delta'<\delta$)
$$\|\ti f_{\underline{k}}\|_{\delta',s}\leq Ce^{-2\pi\rho |\underline{k}|}\frac{1}{(\delta-\delta')^{2d}}\|f\|_{\delta,s}$$
we get 
$$\|f_{\underline{n}}\|_{\delta',s}\leq e^{-4\pi |\underline{n}|\delta} \biggl(\sum_{k\geq 0}\frac{k^n}{n!}e^{-2\pi(\delta-\delta')k}\biggr)^d\frac{1}{(\delta-\delta')^{2d}}\|f\|_{\delta,s}$$
\end{proof}

We define 
$$[f]_{\delta,s}=\max(\|f^{(0)}\|_{\delta,s},\|f^{(1)}\|_{\delta,s},\|\cD\cL f^{(0)}\|_{\delta,s},\|\cD\cL f^{(1)}\|_{\delta,s}).$$
and 
$$M_{f}=\cM(f^{(1)}-f^{(2)}\cD \cL f^{(0)}).$$
\bn{\bf Remark. }Since we shall need it later, we notice that while $\cL$ does not preserve $\s$-symmetry (if $f$ is $\s$-symmetric, $\sqrt{-1}\cL f$ is $\s$-symmetric) as well as $\cD$ (if  $f$ is $\s$-symmetric, $\sqrt{-1}\cD f$ is $\s$-symmetric)  the composition $\cD \cL$ preserves $\s$-symmetry.

Also, we set

$$\{ f\}_{\delta,s}=\max(\max_{0\leq j\leq 2}\|f^{(j)}  \|_{\delta,s},\max_{j=2,3}\|f^{[j]}\|_{\delta,s},\|f\|_{\delta,s})$$
Notice that from Lemma \ref{lem106},
\be \{f\}_{\delta-h,s}\leq Ch^{-3d}\|f\|_{\delta,s}.\label{eq106}\ee
We denote by ${\EE}^{\w,\infty}_{\delta}$ the set of  
(exact) symplectic  (with respect to $dz\wedge dw$) $\s$-symmetric  diffeomorphisms     
defined on $\bD^{2d}_{{\delta}}\times \bD^{d}_{\delta}\times B$
 $$Z_{c,\o}(z,w)=
\left(\begin{array}{l}
z+R(z,w,c,\w) \\
w+T(z,w,c,\omega)
\end{array}\right)$$
with $R,T \in  \cC^{\w,\infty}(\bD_{{\delta}} \times \bD_{\delta}^d \times \bD_{\delta}^d,B)$ and  $R,T=\cO^2(z,w,c)$. If $Z'$ is another mapping in
${\EE}^{\w,\infty}_{\delta}$  then we define
$${\{Z-Z'\}}_{\de,s}=\max(\{R'-R\}_{\de,s},\{T'-T\}_{\de,s}).$$
and
$$(Z\circ Z')_{c,\omega}(z,w)=Z_{c,\o}( Z'_{c,\o}(z,w)).$$
 We denote by ${\EE}^{\w,\sigma,\infty}_{\delta}$ the set of elements of ${\EE}^{\w,\infty}_{\delta}$ that are $\s$-symmetric.

\subsection{Notations}\label{Sec:6.2}
If $h$ is a  positive number  we denote by $\cC_{s}(h)$ an expression of the form $C_{s}\times ( h)^{-\a(s)}$ where $C_{s}$ is a constant  and $\a(\cdot)$  is an increasing real-valued function defined on $\N$. Also, if  $(\e_{s})_{s}$, $(\zeta_{s})_{s}$ are sequences of positive real numbers indexed by $s\in \N$ we use the short hand notation $\<\e,\zeta\>_{s}$ (resp. $\<\e,\e,\zeta\>_{s}$)
to denote the sum of all possible products $\e_{i}\zeta_{j}$ (resp. $\e_{i}\e_{j}\zeta_{k}$) where $i,j$ (resp. $i,j,k$)   take value in $\{0,s\}$ and the value $s$ is taken at most once.

\subsection{The (local) Normal Form Theorem}
Using the preceding notations and the  change of coordinates $z=x+\sqrt{-1}y$, $w=x-\sqrt{-1}y$ we are reduced to prove the following result:

\begin{Thm}\label{theo:5.1} Let $\kappa<1,\delta<1$.  There exist  constants $C,a>0$  (depending only on $\tau$ and $d$) such that if $H(z,w,c)$ is in $\cC^{\omega,\sigma,\infty}_{\delta}$  ($\s$-symmetric) and independent of $\omega$ and 
  if one  assumes that  for some $0<h<\delta/2$
\be\label{smallness}
[H]_{\delta,0}\leq C\biggl((1+\{H\}_{\delta,0})^{-1}\kappa h\biggr)^a  \ee
then, there exist an exact symplectic $\s$-symmetric change of coordinates $Z(c,\omega,z,w)$  in $\cE^{\omega,\s,\infty}_{\delta-h}$ and functions $g(c,\omega,z,w)$, $\Gamma(c,\omega)$, $\Lambda(c,\omega), H'$ in $C^{\omega,\sigma,\infty}_{\delta-h}$,    such that 

\begin{multline*}(H+\<\omega+\Lambda(c,\omega),\cdot\>)\circ Z=\Gamma(c,\omega)+\<\omega,zw-c\>+H'(z,w,c,\omega)+\\ g(z,w,c,\omega)\end{multline*}
where   $g$ is $(\kappa,\tau)$-flat and $[H']_{\delta-h,0}=0$.
Furthermore, $Z,H',g,\Gamma,\Lambda$  satisfy 
\begin{multline}\label{upperbound}
\max\biggl( \|\Lambda\|_{\delta-h,s},\{Z-id\}_{\delta-h,s},\|g\|_{\delta-h,s},\{H'-H\}_{\delta-h,s} \biggr)\\
\leq C_{s}((1+\{H\}_{\delta,0})(\kappa h)^{-1})^{\alpha(s)}[H]_{\delta,0}
\end{multline}
where $C_{s}$ and $\a(s)$ are constants depending only on $s$, $\tau$ and $d$.

Furthermore, if {$(H-H^{[2]})(z,w,c^2)=\cO^{2q+1}(z,w,c)$, then   $(Z-id)(z,w,c^2,\omega)=\cO^{2q}(z,w,c)$,  $g(z,w,c^2,\omega)=\cO^{2q+1}(z,w,c)$ and $\Lambda(c^2,\omega)=\cO^{2q+1}(c)$.

Also,  if  
$$\o_0\in DC(2{\k},\t)$$
then $\bD^{2d}_{\delta'}\times\bD^d_{\delta'}\times \bD^1_{\delta'}\ni (z,w,c,\lambda)\mapsto \Lambda(c,(1+\lambda) \omega_{0}), Z(z,w,c,(1+\lambda) \omega_{0})$ and  $H' (z,w,c,(1+\lambda)\omega_{0})$  are analytic  for some $0<\de'\leq \de$ and 
$g(z,w,c,(1+\lambda)\omega_{0})=0$ on $\bD^{2d}_{\delta'}\times\bD^d_{\delta'}\times \bD^1_{\delta'}$.}
\end{Thm}

\subsection{Proof of the local Normal  Form Theorem}\label{inductive:elliptic}
The proof of Theorem \ref{theo:5.1} is based on the inductive step described in Lemma \ref{indLem}. This lemma is proved in two steps. In a first time we treat the case where 
$M_{H}:=\cM(H^{(1)}-H^{(2)}\cD \cL H^{(0)})$
 is equal to zero and in a second step we show how to reduce to this case by adding a counter term $\<\Lambda(c,\omega),\cdot\>$.

\subsubsection{The case when $M_{H}=0$.}

In the next Lemma we will prove that if $M_H=0$, then one can apply a conjugacy to $H$ to reduce its affine part to a quadratically small one. 

\begin{Lem} \label{sanslambda:elliptic}Let  $H(z,w,c,\omega)\in \cC^{\omega,\s,\infty}_{\delta}$ (hence $\s$-symmetric) and  denote \be \e_{\delta,s}=[H]_{\delta,s},\quad \zeta_{\delta,s}=\{H\}_{\delta,s}+1.\ee
 If $M_{H}=0$ and if  
 \be \e_{\delta,1}\leq \cC_{1}(\kappa h) \zeta_{\delta,1}^{-1}\label{small}\ee
 then
there exist
$Z\in \EE^{\w,\s,\i}_{\de-h}$ ($\s$-symmetric), $\ti \Gamma,\ti H \in \cC^{\w,\sigma,\i}_{\de-h}$ and a $(\k,\t)$-flat function, $\s$-symmetric
$\ti g$ such that 
\begin{equation*}(H+ \langle \w, \cdot \rangle)  \circ Z_{c,\o} (z,w) = \ti\Gamma(c,\omega)+
\langle \w, zw-c\rangle + \ti H(z,w,c,\w)+ \ti g(z,w,c,\w),\end{equation*}
and $$[\ti H]_{\delta-h,s}\leq \cC_{s}(\kappa h)\zeta_{\delta,0}\<\e,\e,\zeta\>_{\delta,s}$$
$$\max\biggl(\|\ti g\|_{\delta-h,s}, \{{\ti H}^{[2]}-H^{[2]}\}_{\delta-h,s}, \{Z-id\}_{\delta-h,s}\biggr)\leq  \cC_{s}(\kappa h) )\zeta_{\delta,0}\<\e,\zeta\>_{\delta,s}$$
{Furthermore, if $(H-H^{[2]})(z,w,c^2,\o)=\cO^{2q+1}(z,w,c)$, then $({\ti H}-{\ti H}^{[2]})(z,w,c^2,\o)=O^{2q+1}(z,w,c)$,  $( Z-id)(z,w,c^2,\omega)=\cO^{2q}(z,w,c)$ and $\ti g(z,w,c^2,\omega)=\cO^{2q+1}(z,w,c)$.}
\end{Lem}

\begin{proof}
Assumme $Z_{k}:(z,w)\mapsto (z',w')$ is an exact symplectic change of variable with generating function $k(z,w',c,\omega)$ depending analytically on $z,w,c$ and smoothly on $\omega$: $z'=z+\pa_{w'}k$, $w=w'+\pa_{z}k$ and denote by $H'$ the hamiltonian defined by $H'(z',w')=H(z,w)$.

With  the  notations of Subsection \ref{subsec:2.2}
\begin{multline*}H(z,w,c,\omega)=H^{(0)}(z,w,c,\omega)+\<zw-c,H^{(1)}(z,w,c,\omega)\>+\\ \phantom{\frac{1}{2}} \<zw-c,H^{[2]}(z,w,c,\omega)(zw-c)\>\end{multline*}
\begin{multline*}H'(z,w,c,\omega)={H'}^{(0)}(z,w,c,\omega)+\<zw-c,{H'}^{(1)}(z,w,c,\omega)\>+\\ \phantom{ \frac{1}{2}}\<zw-c,{H'}^{[2]}(z,w,c,\omega)(zw-c)\>\end{multline*}
with $H^{(0)},H^{(1)}, {H'}^{(0)}, {H'}^{(1)}$ in $\widehat{\cN\cR}$.
The equality $H'(z',w')=H(z,w)$ is equivalent to the fact that 
\begin{multline*}(I):=\<\omega,(z+\pa_{w'}k)w'\>+H^{(0)}(z+\pa_{w'}k,w',c,\omega)+\\ \<(z+\pa_{w'}k)w'-c,H^{(1)} (z+\pa_{w'}k,w',c,\omega) \>+\\ \phantom{\frac{1}{2}}\<(z+\pa_{w'}k)w'-c,H^{[2]}(z+\pa_{w'}k,w',c,\omega)((z+\pa_{w'}k)w'-c)\>
\end{multline*}
is equal to 
\begin{multline*}(II):=\<\omega,z(w'+\pa_{z}k)-c\>+ \Gamma'(c,\omega)+{H'}^{(0)}(z,w'+\pa_{z}k,c,\omega)+\\ \<z(w'+\pa_{z}k)-c,{H'}^{(1)} (z,w'+\pa_{z}k,c,\omega) \>+\\ \phantom{\frac{1}{2}}\<z(w'+\pa_{z}k)-c,{H'}^{[2]}(z,w'+\pa_{z}k,c,\omega)(z(w'+\pa_{z}k)-c)\>+g'(z,w'+\pa_{z}k,c,\omega)
\end{multline*}

Since  $(z+\pa_{w'}k)w'=z(w'+\pa_{z}k)-\cD k$ we can write 
\begin{multline*}0=(I)-(II)=-\cD^\omega k+\<\omega,c\>-\Gamma'(c,\omega)+H^{(0)}(z+\pa_{w'}k,w',c,\omega)+\\ \<(z+\pa_{w'}k)w'-c,H^{(1)} (z+\pa_{w'}k,w',c,\omega) \>+\\ \phantom{\frac{1}{2}}\<z(w'+\pa_{z}k)-c-\cD k,H^{[2]}(z+\pa_{w'}k,w',c,\omega)(z(w'+\pa_{z}k)-c-\cD k)\>\\ -{H'}^{(0)}(z,w'+\pa_{z}k,c,\omega)-\\ \<z(w'+\pa_{z}k)-c,{H'}^{(1)} (z,w'+\pa_{z}k,c,\omega) \>-\\ \phantom{ \frac{1}{2}}\<z(w'+\pa_{z}k)-c,{H'}^{[2]}(z,w'+\pa_{z}k,c,\omega)(z(w'+\pa_{z}k)-c)\>-g'(z,w'+\pa_{z}k,c,\omega)
\end{multline*}
and then,
\begin{multline*}0=(I)-(II)=-\cD^\omega k+\<\omega,c\>-\Gamma'(c,\omega)+H^{(0)}(z,w',c,\omega)+\\ \<zw'-c,H^{(1)} (z,w',c,\omega) \> -2\<\cD k,H^{[2]}(z,w',c,\omega)(zw'-c)\>\\ -{H'}^{(0)}(z,w'+\pa_{z}k,c,\omega)- \<z(w'+\pa_{z}k)-c,{H'}^{(1)} (z,w'+\pa_{z}k,c,\omega) \>-\\ \phantom{\frac{1}{2}}\<z(w'+\pa_{z}k)-c,({H'}^{[2]}(z,w'+\pa_{z}k,c,\omega)-H^{[2]}(z+\pa_{w'}k,w',c,\omega))(z(w'+\pa_{z}k)-c)\> \\ - g'(z,w',c,\omega)+\cQ
\end{multline*}
where $\cQ$ is quadratic expression  in $(H^{(0)},H^{(1)},g',k)$ and their first derivatives and depending on $H^{[2]}$; more precisely 
\begin{multline*}\|\cQ\|_{\delta-h,s}\leq \frac{C_{s}}{h^{3d}}\biggl((\|H^{(0)}\|_{\delta,s}+\|H^{(1)}\|_{\delta,s}+\|H^{[2]}\|_{\delta,s}\|k\|_{\delta,0})\|k\|_{\delta,0}+\\ (\|H^{(0)}\|_{\delta,0}+\|H^{(1)}\|_{\delta,0}+\|H^{[2]}\|_{\delta,0}\|k\|_{\delta,0})\|k\|_{\delta,s}\biggr)
\end{multline*}
Finally,
\begin{multline*}0=(I)-(II)=-\cD^\omega k+\<\omega,c\>-\Gamma'(c,\omega)+H^{(0)}(z,w',c,\omega)+\\ \<zw'-c,H^{(1)} (z,w',c,\omega) \> -2\<\cD k,H^{(2)}(z,w',c,\omega)(zw'-c)\>\\ -{H'}^{(0)}(z,w'+\pa_{z}k,c,\omega)- \<z(w'+\pa_{z}k)-c,{H'}^{(1)} (z,w'+\pa_{z}k,c,\omega) \>-\\ \phantom{\frac{1}{2}}\<z(w'+\pa_{z}k)-c,({H'}^{[2]}(z,w'+\pa_{z}k,c,\omega)-H^{[2]}(z+\pa_{w'}k,w',c,\omega))(z(w'+\pa_{z}k)-c)\> \\+\phantom{\frac{1}{3}}2\<\cD k,H^{[3]}(z,w',c,\omega)(zw'-c)^{\otimes 2}\> - g'(z,w',c,\omega)+\cQ
\end{multline*}

Let us now {\it define} $\Gamma'(c,\omega)$, $k(z,w',c,\omega)=k^{(0)}(z,w',c,\omega)+\<(zw'-c),k^{(1)}(z,w',c,\omega)\>$ and  $g'(z,w',c,\omega)={g'}^{(0)}(z,w',c,\omega)+\<zw'-c,{g'}^{(1)}(z,w',c,\omega)\>$ according to
\be \begin{cases}\Gamma'(c,\omega)&=\<c,\omega\>+\cM H^{(0)}\\
{g'}^{(0)}&=\PP (H^{(0)})\\
{g'}^{(1)}&=\PP(H^{(1)}-\<\cD k^{(0)},H^{(2)}\>)
\end{cases} \label{choiceg'}\ee
\be \begin{cases}
k^{(0)}&=\cL(H^{(0)})\\
k^{(1)}&=\cL(H^{(1)}-\<\cD k^{(0)},H^{(2)}\>)
\end{cases}\label{choicek}\ee
and $Z$, $H'$ by $H'(z',w')=H(z,w)$ and $Z(z,w,c,\omega)=(z',w')$ if and only if  $z'=z+\pa_{w'}k(z,w',c,\omega)$ and $w=w'+\pa_{z}k(z,w',c,\omega)$: 
observe that 
\be \max(  \|k\|_{\delta-h,s},\|g'\|_{\delta-h,s}) \leq \cC_{s}(\kappa h)\<\e,\zeta\>_{\delta,s} \label{6.41}\ee
 and hence, if the latter quantity is small enough (see the comment preceding equation (\ref{eqcom}), by the Inverse Function Theorem (see Proposition 10.3 of \cite{EFK}),  the change of variables $(z,w)\mapsto (z',w')$ and its inverse are well defined. 
Since by assumption $M_{H}:=\cM(H^{(1)}-\<\cD k^{(0)},H^{(2)}\>)=0$ we have 
$$\begin{cases}
\cD^{\omega}k^{(0)}(c,\omega,z,w')&=H^{(0)}(z,w'c,\omega)-(\<\omega,c^2\>-\Gamma'(c,\omega))\\
&-{g'}^{(0)}(z,w',c,\omega)\\
\cD^{\omega}k^{(1)}(c,\omega,z,w')&=H^{(1)}(z,w',c,\omega)-\<\cD k^{(0)},H^{(2)}(z,w',c,\omega)\>\\
&-{g'}^{(1)}(z,w',c,\omega)
\end{cases}$$
We then have 
\begin{multline*}{H'}^{(0)}(z,w'+\pa_{z}k,c,\omega)+ \<z(w'+\pa_{z}k)-c,{H'}^{(1)} (z,w'+\pa_{z}k,c,\omega) \>+\\ \phantom{\frac{1}{2}}\<z(w'+\pa_{z}k)-c,({H'}^{[2]}(z,w'+\pa_{z}k,c,\omega)-H^{[2]}(z+\pa_{w'}k,w',c,\omega))(z(w'+\pa_{z}k)-c)\> \\-\phantom{\frac{1}{3}}2\<\cD k,H^{[3]}(z,w',c,\omega)(zw'-c)^{\otimes 2}\>\\ +2\<\cD k^{(1)}(z,w',c,\omega)(zw'-c),H^{(2)}(z,w',c,\omega)(zw'-c)\>\\ =\cQ
\end{multline*}
If we define 
\begin{multline*}\ti \cQ:={H'}^{(0)}(z,w'+\pa_{z}k,c,\omega)+ \<z(w'+\pa_{z}k)-c,{H'}^{(1)} (z,w'+\pa_{z}k,c,\omega) \>+\\ \phantom{ \frac{1}{2}}\<z(w'+\pa_{z}k)-c,({H'}^{[2]}(z,w'+\pa_{z}k,c,\omega)-H^{[2]}(z+\pa_{w'}k,w',c,\omega))(z(w'+\pa_{z}k)-c)\> \\-\phantom{\frac{1}{3}}2\<\cD k(z,w'+\pa_{z}k,c,\omega),H^{[3]}(z,w'+\pa_{z}k,c,\omega)(z(w'+\pa_{z}k)-c)^{\otimes 2}\>\\ -2\<\cD k^{(1)}(z,w'+\pa_{z}k,c,\omega)(z(w'+\pa_{z}k)-c),H^{(2)}(z,w'+\pa_{z}k,c,\omega)(z(w'+\pa_{z}k)-c)\>
\end{multline*}
we see that  $\ti \cQ$ is still quadratic in $H^{(0)}$, $H^{(1)}$, $g'$,  $k$ and their first derivatives:
\begin{multline}\|\ti \cQ\|_{\delta-h,s}\leq \frac{C_{s}}{h^{3d}}\biggl(\biggl( \|H^{(0)}\|_{\delta,s}+\|H^{(1)}\|_{\delta,s}+\|H^{[2]}\|_{\delta,s}\|k\|_{\delta,0}\\ +\|H^{[3]}\|_{\delta,s} \|k\|_{\delta,0}\biggr) \|k\|_{\delta,0}+\\ \biggl(\|H^{(0)}\|_{\delta,0}+\|H^{(1)}\|_{\delta,0}+\|H^{[2]}\|_{\delta,0}\|k\|_{\delta,0}+\|H^{[3]}\|_{\delta,0}\|k\|_{\delta,0}\biggr) \|k\|_{\delta,s}\biggr)\label{Qtilde}
\end{multline}
and thus
$$\|\ti \cQ\|_{\delta-h,s}\leq \cC_{s}(\kappa h))\zeta_{\delta,0}\<\e,\e,\zeta\>_{\delta,s}.$$
Coming back to the variables $(z,w)$ and setting  $\cQ'(z,w)=\ti\cQ(z,w')$ we see that 
\begin{multline*} {H'}^{(0)}(z,w,c,\omega)+ \<zw-c,{H'}^{(1)} (z,w,c,\omega) \>+\\ \phantom{\frac{1}{2}}\<zw-c,({H'}^{[2]}(z,w,c,\omega)-H^{[2]}(z+\pa_{w'}k,w',c,\omega))(zw-c)\> \\-\phantom{\frac{1}{3}}2\<\cD k(z,w,c,\omega),H^{[3]}(z,w,c,\omega)(zw-c)^{\otimes 2}\>\\ -\<\cD k^{(1)}(z,w,c,\omega)(zw-c),H^{(2)}(z,w,c,\omega)(zw-c)\>=
\cQ'
\end{multline*}
where $\cQ'$ is still quadratic in the following sense: 
from Proposition 10.3 of \cite{EFK} (estimates on composition with the inverse map of the change of variables), (\ref{Qtilde}), (\ref{choicek}) and Lemmas \ref{mainlemma:elliptic} and \ref{lem106} we get, provided $\cC_{1}(\kappa h)\<\e,\zeta\>_{\delta,1}\leq 1$ (which is the case if $\e_{\delta,1}\leq \cC_1(\kappa h) \zeta_{\delta,1}^{-1}$) 
\be  \|\cQ'\|_{\delta-h,s}\leq \cC_{s}(\kappa h)\zeta_{\delta,0}\<\e,\e,\zeta\>_{\delta,s} \label{eqcom}\ee
By Lemma \ref{uniquedecomp},  ${H'}^{(0)}$ and ${H'}^{(1)}$ are uniquely determined by $\cQ'$ since they are in $\widehat{\cN\cR}$ and hence are quadratically small: by Lemma \ref{lem106}
\be [H']_{\delta-h,s}\leq \cC_{s}(\kappa h)\zeta_{\delta,0} \<\e,\e,\zeta\>_{\delta,s}\label{estimee1}\ee
Then, ${H'}^{[2]}-H^{[2]}$ is of the order of (the derivative of) $k$:
\be\{H'^{[2]}-H^{[2]}\}_{\delta-h,s}\leq  \cC_{s}(\kappa h) )\zeta_{\delta,0}\<\e,\zeta\>_{\delta,s}\label{estimee2}\ee

{Finally, in the case $(H-H^{[2]})(z,w,c^2,\omega)=\cO^{2q+1}(z,w,c)$,   formulas (\ref{choiceg'}) and  (\ref{choicek}) show that $g'(z,w,c^2,\omega)$ and $k(z,w,c^2)$  are $\cO^{2q+1}(z,w,c)$. Hence also $(Z-id)(z,w,c^2,\omega)=\cO^{2q}(z,w,c)$.}

We have so far proved that  with the choices (\ref{choiceg'}) and (\ref{choicek})
\begin{equation*}(H+ \langle \w, \cdot \rangle)  \circ Z_{k} (z,w) = \Gamma'(c,\omega)+
\langle \w, zw-c\rangle + H'(z,w,c,\w)+ g'(z,w,c,\w),\end{equation*}
where $H'$ satisfies the estimates (\ref{estimee1}) and (\ref{estimee2}).

We are not completely finished with the proof of our Lemma since nothing insures us that the change of variables $Z_{k}$  we have performed is $\s$-symmetric. Let us introduce $\ti Z_{k}$ the time 1-map of the hamiltonian vector field $\sqrt{-1}J\nabla k$. The equations (\ref{choicek}) show that $\sqrt{-1}k$ is $\s$-symmetric (see the remark following Lemma \ref{lem106}) and thus $\ti Z_{k}$ is $\s$-symmetric (see Lemma \ref{lem71}). The assumption (\ref{small}) allows to apply  Proposition \ref{PropApp}: we have $\{Z_{k}^{-1}\circ \ti Z_{k}-I\}_{\delta-h,s}=\cC_{s}(h)\<\|k\|,\|k\|\>_{\delta,s}$ and thus we can write 
\be (H+ \langle \w, \cdot \rangle)  \circ \ti Z_{k}=  \ti \Gamma(c,\omega)+
\langle \w, zw-c\rangle + \ti H(z,w,c,\w)+ g'(z,w,c,\w).
\ee
where the  estimates on composition of Proposition 10.2 of \cite{EFK} and estimates (\ref{6.41}) show that 
\be [\ti H]_{\delta-h,s}\leq \cC_{s}(\kappa h)\zeta_{\delta,0} \<\e,\e,\zeta\>_{\delta,s}\label{estimee1ti}\ee
\be\{\ti H^{[2]}-H^{[2]}\}_{\delta-h,s}\leq  \cC_{s}(\kappa h) )\zeta_{\delta,0}\<\e,\zeta\>_{\delta,s}.\label{estimee2ti}\ee
Equations (\ref{choiceg'}), show that $g'$ is $\s$-symmetric.

\end{proof}
\subsubsection{ Elimination of the mean value $M_{H}$}
Here is a lemma similar to Lemma 8.4 of \cite{EFK}, that allows to eliminate $M_H$ by adding a term $\langle \Lambda, \cdot \rangle$.

\begin{Lem} \label{lambda:elliptic}Let $W\in \cE_{\delta}^{\omega,\s,\infty}$ and denote $$\eta_s=\{W-\id\}_{\delta,s}.$$ There exists a constant $\cC_{0}$ such that if $\eta_{\delta,0}\leq \cC_{0}\zeta_{\delta,0}^{-1}$
then there exists $\Lambda \in \CC^{\w,\i}_{\de}$, $\Lambda=\Lambda(c,\omega)$  such that
$$ \tilde H_{\Lambda}=H+\langle \Lambda, \cdot \rangle \circ W$$
verifies $ M_{\tilde H} = 0$ and 
such that  for all $s\in\N$, $0<h<\delta$
\begin{equation*} 
\aa{\Lambda}_{\delta-h,s}\leq  \cC_s(\kappa h)\zeta_{\delta,0}\<\e,\zeta,\eta+\zeta+\e\>_{\delta,s} \end{equation*}
and
\begin{equation*} 
[\tilde H_{\Lambda}-H_{\Lambda}]_{\delta-h,s} \leq \cC_s(\kappa h)\zeta_{\delta,0}\<\e,\zeta,\eta+\zeta+\e\>_{\delta,s}.\end{equation*}
{Furthermore, if $(H-H^{[2]})(z,w,c^2,\omega)=\cO^{2q+1}(z,w,c)$ then $\Lambda(c^2,\omega)=\cO^{2q+1}(c)$ and $(\ti H_{\Lambda}-{\ti H}_{\Lambda}^{[2]})(z,w,c^2,\omega)=\cO^{2q+1}(z,w,c)$.}
\end{Lem}
\begin{proof}Let us denote by $W_{c,\omega}:(z,w)\mapsto (z',w')=(z+R(z,w,c,\omega),w+T(z,w,c,\omega))$. Using $\ti H_{\Lambda}=H+\<\Lambda(c,\omega),(z+R)(w+T)-c)\>$, we now compute the canonical decomposition of $\ti H_{\Lambda}$ in terms of the canonical decomposition of $H$  and $\Lambda$. 
We have
\begin{multline*}
\<\Lambda,(z+R)(w+T)-c)\>=\<\Lambda,zw-c\>+\sum_{j=0}^2\phantom{\frac{1}{j!}}\<\Lambda,U^{(j)}(zw-c)^{\otimes j}\>\\ +\phantom{\frac{1}{6}}\<\Lambda,U^{[3]}(zw-c)^{\otimes 3}\>
\end{multline*}
where $\sum_{j=0}^2 U^{(j)}(zw-c)^{\otimes j}+U^{[3]}(zw-c)^{\otimes 3}$ is the canonical decomposition of $zT+wR+RT$. Then, 
for $j=0,2$, ${\ti H}_{\Lambda}^{(j)}=H^{(j)}+\<\Lambda,U^{(j)}\cdot\>$ and $\ti H_{\Lambda}^{(1)}= H^{(1)}+\<\Lambda,(I+U^{(1)})\cdot\>$.
From this it follows that $\Lambda\mapsto M_{\ti H_{\Lambda}}=\cM(\ti H_{\Lambda}^{(1)}-\ti H_{\Lambda}^{(2)}\cD\cL \ti H_{\Lambda}^{(0)})$ is a map of the form $M_{H}+a_{1}\cdot \Lambda+a_{2}\cdot(\Lambda\otimes\Lambda)$ with 
$$\|M_{H}\|_{\delta-h,s}\leq \cC_{s}(\kappa h)\zeta_{\delta,0}\<\e,\zeta\>_{\delta,s}.$$
$$\max (\|a_{1}-I\|_{\delta-h,s},\|a_{2}\|_{\delta-h,s}) \leq \cC_{s}(\kappa h)\zeta_{\delta,0}\<\eta+\zeta+\epsilon,\eta\>_{\delta,s}$$
Now,  the first part  of the lemma follows from (\ref{eq106}), the Inverse Function Theorem and the estimates of Section 10  of \cite{EFK}.  

{Furthermore, since $M_{H}(c^2,\omega)=\cO^{2q+1}(c)$, we see that $\Lambda(c^2,\omega)=\cO^{2q+1}(c)$ and the last statement of the lemma is proven}
\end{proof}

\subsubsection{The inductive step} Putting together Lemmas \ref{sanslambda:elliptic} and  \ref{lambda:elliptic}  we get similarly to Proposition 8.2 of \cite{EFK} the following KAM induction step.
\begin{Lem}\label{indLem}Let  $H,g,\Gamma$ be in $\cC_{\rho,\delta}^{\omega,\s,\infty}$ where $g$ is $(\kappa,\tau)$-flat  and $W$ be in $\cE_{\delta}^{\omega,\s,\infty}$. There are positive  constants $C$ and $a$ such that if $\eta_{0}\leq C\zeta_{0}^{-1}$ and $\e_{1}\leq\biggl( \kappa h \zeta_{\delta,1}^{-1}(1+\eta_{\delta,1})^{-1}\biggr)^{a}$ then there exist $Z'\in \EE^{\w,\s,\i}_{\de-h}$, $\Gamma',\Lambda',g', H' \in \cC^{\w,\s,\i}_{\de-h}$ where $g'$ is  $(\k,\t)$-flat such that
 \begin{multline*}(H+\Gamma(c,\omega)+g+ \langle \w, \cdot \rangle+(\<\Lambda'(c,\omega),\cdot\>)\circ W)  \circ Z'_{c,\o} (z,w) =\\ \Gamma'(c,\omega)+
\langle \w, zw-c\rangle + H'(z,w,c,\w)+ g'(z,w,c,\w),\end{multline*}
and for any $s\in\N$, $0<h<\delta/2$, 
$$[H']_{\delta-h,s}\leq \cC_{s}(\kappa h)\zeta_{\delta,0}\<\e,\e,\zeta\>_{\delta,s}$$
\begin{multline*}\max\bigg(\|\Lambda\|_{0,\delta-h,s},\|g'-g\|_{\delta-h,s}, \{H'^{[2]}-H^{[2]}\}_{\delta-h,s},\\ \{Z'-id\}_{\delta-h,s},\{W\circ Z'-W\}_{\delta-h,s}\biggr)\leq  \cC_{s}(\kappa h)\zeta_{\delta,0}^{a}\<\e,\zeta+\eta\>_{\delta,s}\end{multline*}
{Furthermore, if $(H-H^{[2]})(z,w,c^2,\omega)=\cO^{2q+1}(z,w,c)$ and $g(z,w,c^2,\omega)=\cO^{2q+1}(z,w,c)$, then $(H'-{H'}^{[2]})(z,w,c^2,\omega)=\cO^{2q+1}(z,w,c)$ and $(Z'-id)(z,w,c^2,\omega)=\cO^{2q}(z,w,c)$, $\Lambda'(c^2,\omega) =\cO^{2q+1}(c)$, $g'(z,w,c^2,\omega)=\cO^{2q+1}(z,w,c)$.}
\end{Lem}

\subsubsection{Convergence of the KAM scheme} As in  Section 8.5 of \cite{EFK} the preceding Lemma \ref{indLem} applied inductively is enough to prove  Theorem \ref{theo:5.1}. We refer the reader to Sections 8.4 and 8.5 of \cite{EFK} for a proof of this fact. 

{\subsubsection{End of the proof} To this point we have proven  a theorem, let us call it $(T')$,  which is Theorem \ref{theo:5.1} except the statement on  the analyticity with respect to $\lambda$  when $\omega$ is replaced by  $(1+\lambda)\omega_{0}$.  This theorem $(T')$  applied   to  the analytic function $(z,w,(c,\lambda))\mapsto H(z,w,c,(1+\lambda)\omega_{0})$, $(z,w,c,\lambda)\in \bD_{\delta}^{2d}\times \bD_{\delta}^{d+1}$,  with $s=0$,   completes the proof of Theorem \ref{theo:5.1}.}

\section{Appendix }
\subsection{Appendix A}

For $\kappa>0$  we assume given for some $\eta_\k>0$, with $\lim_{\kappa  \to 0} \eta_\kappa = 0$, and a family of maps $W_\kappa: \R^d \times \T^d \to \R^d \times \R^d$ that are of the form  
\begin{align*} W_\kappa(c,\theta)&=W(c,\th)+\eps_\kappa \\
W(c,\th)&=\left(c_1 \sin(2\pi \th_1),c_1 \cos(2\pi \th_1),\ldots,c_d \sin(2\pi \th_d),c_d \cos(2\pi \th_d)\right)\end{align*}
where $\eps_\kappa(0,\cdot)\equiv0$   we have for every $\xi\leq  \eta_\k$
\begin{equation} \label{jac}  \max_{x \in \cC_d(\xi) \times \T^d} \left( | \eps_\kappa (x)|+|  \nabla \eps_\kappa (x)| \right) < \xi^{2d+1} \end{equation}
where $\cC_d(\xi):= \{ c \in \R^d : |c_i|<\xi, \forall i\}$, and assume that  $\Sigma_\kappa \subset \cC_d (\eta_\k)$ are a family of measurable sets such that 
$$\lim_{\kappa \to 0} \frac{\mes(\Sigma_{\kappa})}{\mes(\cC_{d}(\eta_\k))} =1.$$

Denote $B_{2d}(0,\xi)$  the product $\{x_1^2+y_1^2\leq \xi \} \times \ldots \times \{x_d^2+y_d^2 \leq \xi \}$.
  
Then we have the following 
\begin{lemma}\label{lemmaA1}  Denote by $\widetilde{\Sigma}_{\kappa}=W_\kappa(\Sigma_\kappa \times \T^d)$. Then, for any $\nu>0$, if $\k$ is sufficiently small we have
\begin{equation} \label{ed1} \mes(\widetilde{\Sigma}_{\k} \cap B_{2d}(0,\eta_\k))/\mes(B_{2d}(0,\eta_\k)) > 1-\nu\end{equation}
\end{lemma}

\begin{proof}  For $\eps>0$, define $\cC_d(\eta_\k,\eps):=  \cC_d(\eta_\k)  \cap \{ |c_i|> \eps |c_j|, \forall i,j \}$. 
We also define $\Sigma_{\kappa,\eps}=\Sigma_\kappa \cap \cC_d(\eta_\k,\eps)$.

We have that $W(\cC_d(\eta_\k) \times \T^d)=B_{2d}(0,\eta_\k)$. Also, it is not hard to see that if $\eps$ and  then $\kappa$ are sufficiently small then
\begin{equation} \label{ed0} \mes(W(\Sigma_{\k,\eps} \times \T^d))/\mes(B_{2d}(0,\eta_\k)) > 1-\nu^2\end{equation}
and from \eqref{jac}
\begin{equation*} \label{ed3} | \Jac W_\k - \Jac W | < \nu^2 |\Jac W|\end{equation*}
on  $\cC_d(\eta_\k,\eps) \times \T^d$,  which gives
\begin{equation}\label{ed4} \mes(W_\k(\Sigma_{\k,\eps} \times \T^d )) > (1-\nu^2) \mes(W(\Sigma_{\k,\eps} \times \T^d ))
\end{equation}
Wa also have that $W_\k(\cC_d(\eta_\k)\times \T^d ) \subset B_{2d}(0,\eta_\k+o(\eta_\k))$.
 \eqref{ed1} hence follows from \eqref{ed0} and \eqref{ed4}  if $\nu \ll 1$. 

\end{proof}

\subsection{Appendix B: Generating functions and time-1 map of Hamiltonian flows}
There are two classical methods to construct symplectic diffeomorphisms. The first one, which we have been using throughout the paper, is the generating function method: given $f:(\bC^{2d},0)\to\C$ we define the symplectomorphism $Z:(\bC^{2d},0)\to (\bC^{2d},0)$ {\it implicitly} by the equations 
\be\begin{cases}\ Z_{f}(z,w)=(z',w') 
\end{cases}
 \iff  \label{2.5} \begin{cases}
z'=z+\pa_{w'}f(z,w')\\
w=w'+\pa_{z}f(z,w').
\end{cases}  \ee

A second classical method is to use the so-called ``Lie method''. Given $f:(\bC^{2d},0)\to\C$, we introduce the hamiltonian flow $(\phi_{f}^t)$ defined by the hamiltonian (with respect to the symplectic form $dz\wedge dw$) vector field $J\nabla f$ where $J=\begin{pmatrix}0& -I_{d}\\I_{d}& 0\end{pmatrix}$ and we let $\ti Z_{f}$ be the time-1-symplectomorphism $\phi^1_{f}$.

The first method is well adapted in the formal setting but has  the drawback of not preserving $\s$-symmetry. On the other hand, the Lie method preserves $\s$-symmetry.
\begin{Lem}\label{lem71}If $\sqrt{-1}f$ is $\s$-symmetric, then $\ti Z_{f}$ is $\s$-symmetric.
\end{Lem}
\begin{proof}Let us use the  change of variables $(z,w)\mapsto (x,y)$, $z=\frac{1}{2}(x+\sqrt{-1}y)$, $w=\frac{1}{2}(x-\sqrt{-1}y)$. It transforms the symplectic form $dz\wedge dw$ to $-\sqrt{-1}dx\wedge dy$ and the hamiltonian flow $J\nabla f(z,w)$ is transported to $\sqrt{-1}J\nabla \ti f(x,y)$ where $\ti f(x,y)=f(z,w)$. But $\sqrt{-1}\ti f(x,y)$ takes real values when $x$ and $y$ are real (since $\sqrt{-1}f$ is $\s$-symmetric). Hence its time-1-map has the same property. Coming back to the variables $(z,w)$ shows that $\ti Z_{f}$ is $\s$-symmetric.
\end{proof}

Nevertheless, we notice that, $Z_{f}$ and $\ti Z_{f}$ differ by a quantity which is quadratic in $f$ (and its derivatives).

\begin{Prop}\label{PropApp}There exists $\xi>0$ such that if  $f\in \cC^{\omega,\infty}_{\delta}$ satisfies for $0<h<\delta$
\be \|f\|_{\delta,1}\leq \xi h^2
\ee
then \footnote{The notations are those of Section \ref{Sec:6.2}. },
\be\{\ti Z_{f}^{-1}\circ Z_{f}-Id\}_{\delta-h,s}=\cC_{s}(h)\<\|f\|,\|f\|\>_{\delta,s}\ee
Furthermore, if $f(z,w,c^2,\omega)=\cO^{2q+1}(z,w,c)$ then $\ti Z_{f}^{-1}\circ Z_{f}-id=\cO^{2q}(z,w,c)$.
\end{Prop}
\begin{proof}
Let us denote by $W$ the local diffeomorphism $(\bC^{2d},0)\to\bC^{2d}$, $W(z,w)=(z+\pa_{w}f(z,w,c,\omega),w-\pa_{z}f(z,w,c,\omega))$.
Proposition \ref{PropApp} will follow from the following two lemmas \ref{AppL1} and \ref{AppL2}
\begin{Lem}\label{AppL1}There exists a constant $\xi>0$ such that if  $0<h<\delta$ and $f$ satisfy 
\be \|f\|_{\delta,1}\leq \xi h^2\label{aass}\ee
then
\be\{W- Z_{f}\}_{\delta-h,s}=\cC_{s}(h)\<\|f\|,\|f\|\>_{\delta,s}\ee
\end{Lem}
\begin{proof}
If $Z_{f}:(z,w)\mapsto (z',w')$ we can write 
\be\begin{cases}
z'=z+\pa_{w}f(z,w,c,\omega)+ (\pa_{w'}f(z,w',c,\omega)-\pa_{w}f(z,w,c,\omega))\\
w'=w-\pa_{z}f(z,w,c,\omega)+(\pa_{z}f(z,w,c,\omega)-\pa_{z}f(z,w',c,\omega)).
\end{cases}  \ee
and using Proposition (10.3)  of \cite{EFK} we  notice that  $(z,w)\mapsto w'(z,w,c,\omega)-w$  has a $\|\cdot\|_{\delta-h,s}$ norm less or equal than $\cC_{s}(h)\|f\|_{\delta,s}$ and by Proposition 10.2 of \cite{EFK} that $(z,w)\mapsto f(z,w',c,\omega)-f(z,w,c,\omega)$ has a $\|\cdot\|_{\delta-h,s}$ norm less or equal than $\cC_{s}(h) \< \|f\|,\|f\|\>_{\delta,s}$. The conclusion then follows.
\end{proof}

\begin{Lem} \label{AppL2}There exists a constant $\xi>0$ such that if  $0<h<\delta$ and $f$ satisfy 
\be \|f\|_{\delta,0}\leq \xi h^2\label{aass}\ee
then  one has 
\be\{W- \ti Z_{f}\}_{\delta-h,s}=\cC_{s}(h)\<\|f\|,\|f\|\>_{\delta,s}\ee
\end{Lem}
\begin{proof}
If  $u_{u_{0},c,\omega}(\cdot)$ (we shall denote $u$ for short, $u(t):=(z(t),w(t))$) is a solution of the differential equation $u'(t)=J\nabla f(u(t),c,\omega)$, $u(0)=u_{0}$, $|u(0)|\leq \delta-h$, $|c|\leq \delta$, $\omega\in B$ one has, as long as the solution $u(\cdot)$  is defined,
\be u(t)=u(0)+\int_{0}^tJ\nabla f(u(s),c,\omega)ds,\label{a1}\ee

Let $[0,t_{max})$ a maximal interval of definition of the solution $u$ and,  if it exists,   $t_{*}:=\inf\{t\in [0,t_{max}): \ |u(t)|>\delta-(h/2)\}$. One has, for $0\leq t< t_{*}$ and some constant $C$
\begin{align} |u(t)|&\leq |u(0)|+C t\|J\nabla f\|_{\delta,0}\\
&\leq |u(0)|+Ct(h/2)^{-1}\|f\|_{\delta,0}.
\label{a2}\end{align}
Assume that $t_{*}$ exists and is $\leq 1$; then 
\begin{align*} |u(t_{*})| &\leq \delta-h+ Ct(h/2)^{-1}\|f\|_{\delta,0}\\
&\leq \delta-(3h/4) \end{align*}
provided $C(h/2)^{-1}\|f\|_{\delta,0}\leq h/4$,  which is the case if  the constant $\xi$ in (\ref{aass}) is small enough. But, by definition $|u(t_{*})| \geq \delta-(h/2)$ which is a contradiction; hence $t_{*}$ if it exists is $>1$.

The theorem on continuous (and differentiable) dependence of the solution of an O.D.E with respect to the initial condition and parameters then shows that 
 $(z,w)\mapsto \phi^1_{f}(z,w,c,\omega)$ is an analytic diffeomorphism  with respect to $(z,w)\in \bD_{\delta-h}^{2d}$, analytic with respect to $c\in \bD_{\delta}^{2d}$ and depending smoothly on $\omega \in B$.

Now, the Linearization Theorem for O.D.E. tells us that the derivative $v(\cdot)=\pa_{\omega} u_{u_{0},c,\omega}(\cdot)$  of $u_{u_{0},c,\omega}(\cdot)$ with respect to $\omega$ satisfies the affine equation 
\be  v'(t)=D_{u}J\nabla f(u_{u_{0},c,\omega}(t),c,\omega)\cdot v(t)+D_{\omega} J\nabla f(u_{u_{0},c,\omega}(t),c,\omega)\ \ee
with initial condition $v(0)=0$. 
More generally, $\pa_{\omega}^\a u_{u_{0},c,\omega}$ satisfies the differential equation
\be \frac{d}{dt}  \pa_{\omega}^\a u_{u_{0},c,\omega}(t)=D_{u}J\nabla f(u_{u_{0},c,\omega}(t),c,\omega)\cdot \pa_{\omega}^\a u_{u_{0},c,\omega}(t)+G_{\a}(t,u_{0},c,\omega) \label{a5}\ee
with initial condition $\pa_{\omega}u_{u_{0},c,\omega}(0)=0$ and 
where $G_{\a}$ is a finite sum of terms (the number of which depends only on $|\a|$ and $d$)  of the form 
\begin{multline}D_{u}^mD_{\omega}^l J\nabla f(u_{u_{0},c,\omega}(t),c,\omega)\cdot (\pa_{\omega}^{\beta_{1}}u_{u_{0},c,\omega}(t),\ldots,\pa_{\omega}^{\beta_{m}}u_{u_{0},c,\omega}(t))\label{a6}
\end{multline}
with $|\beta_{1}|+\cdots+|\beta_{m}|+l=|\a|$, $(m,l)\ne (1,0)$.
Let us now prove by induction on $|\a|$ that for any $0\leq t\leq 1$
\be \|\pa_{\omega}^\alpha u_{\cdot}(t)\|_{\delta-h,0}\leq \cC_{|\a|}(h(1+||f||_{0,\delta})^{-1})\|f\|_{\delta,|\a|}.\label{aind}\ee
Assume that there exist a positive valued  increasing function $s\mapsto a(s)$ defined for $s\in\N$, $0\leq s\leq |\a|-1$  and $C>0$ such that for any $|\beta|<|\alpha|$, any $0\leq t\leq 1$
\be \|\pa_{\omega}^\beta u_{\cdot}(t)\|_{\delta-h,0}\leq Ch^{-a(|\beta|)}(1+\|f\|_{\delta,0})^{a(|\beta|)} \|f\|_{\delta,|\beta|} .\label{a7}\ee
We get for $|u_{0}|\leq \delta-h$, $|c|\leq \delta$, $\omega\in B$, $0\leq t\leq 1$
\begin{multline*}|(\ref{a6})|\leq C_{\a}\delta^{-m-1}\|D_{\omega}^l f\|_{\delta,0} (Ch^{-a(|\beta_{1}|)}(1+\|f\|_{\delta,0})^{a(|\beta_{1}|)}\|f\|_{\delta,|\beta_{1}|})\cdots \\(Ch^{-a(|\beta_{m}|)}(1+\|f\|_{\delta,0}))^{a(|\beta_{m}|)}\|f\|_{\delta,|\beta_{m}|}),
 \end{multline*}
and using the convexity estimates, see Proposition 10.1 of \cite{EFK} 
\begin{multline*}|(\ref{a6})|\leq C_{\a}\delta^{-m-1}\|f\|_{\delta,0}^{1-l/|\a|}\|f\|_{\delta,|\a|}^{l/|\a|}  \\ (C_{\beta_{1}}h^{-a(|\beta_{1}|)|} (1+\|f\|_{\delta,0})^{a(|\beta_{1}|)}\|f\|_{\delta,0}^{1-|\beta_{1}|/|\a|}\|f\|_{\delta,|\a|}^{|\beta_{1}|/|\a|})\cdots \\ (C_{\beta_{m}}h^{-a(|\beta_{m}|)}(1+\|f\|_{\delta,0})^{a(|\beta_{m}|)}\|f\|_{\delta,0}^{1-|\beta_{m}|/|\a|}\|f\|_{\delta,|\a|}^{|\beta_{m}|/|\a|}) \end{multline*}
and finally since $|\beta_{1}|+\cdots+|\beta_{m}|+l=|\a|$
\be |(\ref{a6})|\leq C_{\a}(h^{-1}(1+\|f\|_{\delta,0}))^{a(|\a|)}\|f\|_{\delta,|\a|}\label{a8} \ee
provided $a(|\a|)\geq m+1+a(|\beta_{1})|+\cdots+a(|\beta_{m}|)$.
Let us come back to the affine differential  equation (\ref{a5}) and let  $R(t,s)$ be the resolvent of the associated  linear differential equation,
\be  v'(t)=D_{u}J\nabla f(u_{u_{0},c,\omega}(t),c,\omega)\cdot v(t).\ee
By the variation of constant formula we get 
$$\pa_{\omega}^\a u_{u_{0},c,\omega}(t)=\int_{0}^tR(t,s) G_{\a}(t,u_{0},c,\omega) ds.$$
For $u_{0}\in \bD_{\delta-h}^{2d}$, $c\in \bD_{\delta}^{2d}$, $\omega\in B$, $0\leq s\leq t\leq 1$  we see that $|R(t,s)|\leq e^{M}$ where $M$ is the supremum of the norm of $D_{u}J\nabla f(u,c,\omega)$ on $\bD_{\delta-h/2}\times \bD_{\delta}\times B$. We notice that  $M\leq {\rm const.} \|f\|_{\delta,0}h^{-2}$ and that if the constant $\xi$ in (\ref{aass}) is small enough $M\leq 1$.
Hence, for $0\leq t\leq 1$, 
we get from (\ref{a8})
\be |\pa_{\omega}^\a u_{u_{0},c,\omega}(t)| \leq  C_{\a}e (h^{-1}(1+\|f\|_{\delta,0}))^{a(|\a|)}\|f\|_{\delta,|\a|}.\ee
This complete the proof of (\ref{aind}) by induction.

To finish the proof of the Lemma we write
\be u_{u_{0},c,\omega}(1)-u_{0}=\int_{0}^1(J\nabla f(u_{u_{0},c,\omega}(s),c,\omega)-J\nabla f(u_{0},c,\omega))ds
\ee
and use Proposition 10.2 (i) of \cite {EFK}:
$$\|u_{\cdot}(1)-\cdot\|_{\delta-h,s}\leq h^{-1}\int_{0}^1\<\|J\nabla f\|,u_{\cdot}(t)\>_{\delta,s}dt$$
which is $\leq \cC_{s}(h(1+\|f\|_{\delta,0}^{-1})\<\|f\|,\|f\|\>_{\delta,s}$.
\end{proof}
The proof of Proposition \ref{PropApp} can now be completed using the estimates of Proposition 10.2 and 10.3 of \cite{EFK} on compositions and inverses of functions. The last statement of the Proposition follows from the validity of a similar statement in Lemma \ref{AppL1} and \ref{AppL2}.
\end{proof}

\end{document}